\documentclass[12pt,reqno]{amsart}
\usepackage{amsmath,amsthm,amsfonts,amssymb,amscd,amstext}
\usepackage[ansinew]{inputenc}
\usepackage[dvips]{graphicx}
\usepackage{psfrag,euscript}

\usepackage{pxfonts}   
\numberwithin{equation}{section}

\newcommand {\cal} {\mathcal}

\theoremstyle{plain}
\newtheorem{maintheorem}{Theorem}

\newtheorem{theorem}{Theorem}[section]
\newtheorem{proposition}[theorem]{Proposition}

\newtheorem{lemma}[theorem]{Lemma}

\theoremstyle{definition}
\newtheorem{remark}[theorem]{Remark}

\newcommand{\RR}{{\mathbb R}}

\newcommand{\NN}{{\mathbb N}}
\newcommand{\ZZ}{{\mathbb Z}}

\newcommand{\ov}{\overline}
\newcommand{\vfi}{\varphi}

\newcommand{\var}{\operatorname{var}}

\renewcommand{\epsilon}{\varepsilon}

\newcommand{\cE}{\EuScript{E}}

\newcommand{\I}{\EuScript{I}}
\newcommand{\cP}{\EuScript{P}}

\newcommand{\Q}{\EuScript{Q}}

\newcommand{\U}{\EuScript{U}}

\newcommand{\cC}{\EuScript{C}}

\newcommand{\M}{\EuScript{M}}
\newcommand{\A}{\EuScript{A}}
\newcommand{\B}{\EuScript{B}}
\newcommand{\R}{\EuScript{R}}

\newcommand \lap {\lambda^{\prime}}
\newcommand \pip {\pi^{\prime}}

\newcommand \la {\lambda}

\newcommand{\modk}{{\mathcal M}_{\kappa}}
\newcommand \HH {{\mathcal H}}
\newcommand \VR {{\mathcal V}_0^{(1)}(\R)}
\newcommand \cR {{\mathcal R}}

\newcommand \Y {{\mathcal Y}}
\newcommand \cY {{\mathcal Y}}

\title[Large deviations for Teichm\"uller]
{A Large Deviations Bound for the Teichm\"uller Flow on the Moduli Space of Abelian Differentials}

\author{V\'{\i}tor Ara\'ujo}

\address{ V\'\i tor Ara\'ujo, Instituto de Matem\'a\-tica,
  Universidade Federal do Rio de Janeiro, C. P. 68.530,
  21.945-970 Rio de Janeiro, RJ-Brazil}
\email{vitor.araujo@im.ufrj.br}

\author{Alexander I. Bufetov}

\address{Alexander I. Bufetov, Department of Mathematics,
  Rice University, MS 136, 6100 Main Street, Houston, Texas
  77251-1892 \textrm{and} The Steklov Institute of
  Mathematics, Russian Academy of Sciences, Gubkina str. 8,
  119991, Moscow, Russia} \email{aib1@rice.edu \textrm{and}
  bufetov@mi.ras.ru}

\date{\today}

\begin{document}

\subjclass{
37D25,
37A50, 37B40, 37C40}

\renewcommand{\subjclassname}{\textup{2000} Mathematics Subject Classification}

\keywords{suspension flows, moduli spaces, large deviations,
  Gibbs equilibrium states, countable shifts}

\maketitle

\begin{abstract}
  Large deviation rates are obtained for suspension flows
  over symbolic dynamical systems with a countable
  alphabet. The method is that of the first
  author~\cite{araujo2006a} and follows that of
  Young~\cite{Yo90}. A corollary of the main results
  is a large deviation bound for the Teichm{\"u}ller flow on
  the moduli space of abelian differentials, which extends
  earlier work of J. Athreya~\cite{Athreya06}.
\end{abstract}


\section{Introduction}
\label{sec:introd}
\subsection {The Teichm{\"u}ller flow.}

Let $g\geq 2$ be an integer. Take an arbitrary integer
vector $\kappa=(k_1, \dots, k_{\sigma})$ such that $k_i>0$,
$k_1+\dots +k_{\sigma}=2g-2$.

Let $\modk$ be the moduli space of abelian differentials
with singularities prescribed by $\kappa$, or, in other
wors, the moduli space of pairs $(M, \omega)$ such that $M$
is a compact oriented Riemann surface of genus $g$ and
$\omega$ is a holomorphic one-form on $M$ whose zeros have
orders $k_1, \dots, k_{\sigma}$.  We impose the additional
normalization requirement
$$
\frac 1{2i}\int_M \omega\wedge {\overline \omega}=1.
$$
(in other words, the surface $M$ has area $1$ with respect
to the area form induced by $\omega$).  The space $\modk$
need not be connected and we denote by $\HH$ a connected
component of $\modk$.  The Teichm{\"u}ller flow $g_t$ on
$\HH$ is defined by the formula
$$
g_s(M, \omega)=(M^{\prime}, \omega^{\prime}), \ {\mathrm {where}} \ \omega^{\prime}=e^t\Re(\omega)+i e^{-t}\Im(\omega),
$$
and the complex structure on $M^{\prime}$ is uniquely
determined by the requirement that the form
$\omega^{\prime}$ be holomorphic.

The flow $g_t$ preserves a natural ``smooth'' probability
measure on $\HH$ (Masur \cite{Masur82}, Veech
\cite{veech82}), which we denote by $\mu_{\kappa}$ (see,
e.g., \cite{kz1} for a precise definition of the smooth
measure; informally, the construction of $\mu_{\kappa}$ can
be explained as follows: by the Hubbard-Masur Theorem
\cite{hubmas}, the relative periods of $\omega$ with respect
to its zeros yield a local system of coordinates on $\HH$;
up to a scalar multiple, the measure $\mu_{\kappa}$ is
simply the Lebesgue measure in the Hubbard-Masur
coordinates).

Veech \cite{veech86} proved that the Teichm{\"u}ller flow is
a Kolmogorov flow with respect to $\mu_{\kappa}$.
Furthermore, $\mu_{\kappa}$ is the unique measure of maximal
entropy for the flow $g_t$ \cite{BufGur}.

In fact, the Teichm{\"u}ller flow preserves a pair of
infinitely smooth stable and unstable foliations on $\HH$,
the measure $\mu_{\kappa}$ admits globally defined
conditional measures on the stable and unstable leaves, and
the flow $g_t$ expands and contracts the conditional
measures uniformly. Informally, $\mu_{\kappa}$ is the
Bowen-Margulis measure for $g_t$.

Veech \cite{veech86} showed that the flow $g_t$ admits no
zero Lyapunov exponents with respect to the smooth measure
(and all ergodic measures satisfying a technical condition).
Forni \cite{forni02} showed that the expansion on the
unstable leaves (as well as contraction on stable leaves) is
{\it uniform on compact sets} (whence, in particular,
absence of zero exponents for the flow follows for all
ergodic measures).

Furthermore, the Teichm{\"u}ller flow satisfies the
following exponential estimate for visits into compact sets.
Let $K\subset \HH$ be a compact set with nonempty interior.

For $X\in\HH$ set
$$
\tau_K(X)=\{\inf t: g_tX\in K\}.
$$

Then there exists $\alpha>0$ such that
\begin{equation}
\label{expretest}
\int_{\HH} \exp(\alpha \tau_K(X))d\mu_{\kappa}(X)<+\infty.
\end{equation}

J.Athreya established the estimate (\ref{expretest}) for a
special family of ``large'' compact sets $K$.  For arbitrary
compact sets with nonempty interior (in fact, it suffices to
require for some $t_0>0$ that the interior of the set $
\cup_{0\leq t\leq t_0} g_tK $  be nonempty) the
exponential estimate was proved in \cite{bufetov06} and
independently by Avila, Gou{\"e}zel and Yoccoz in
\cite{agy}.


The uniform hyperbolicity of the Teichm{\"u}ller flow on
compact sets in combination with the estimate
(\ref{expretest}) allow one to carry over to the
Teichm{\"u}ller flow a number of facts known about geodesic
flows on compact manifolds of negative curvature.  In
particular, in \cite{bufetov06} it is shown that the
Teichm{\"u}ller flow satisfies the Central Limit Theorem
with respect to $\mu_{\kappa}$, while Avila, Gou{\"e}zel and
Yoccoz in \cite{agy} have shown that the time correlations
of the Teichm{\"u}ller flow decay exponentially.  This paper
is devoted to large deviations for the Teichm{\"u}ller flow.

Take $\delta>0$, let $\varphi:\HH\to {\mathbb R}$ have
average zero and consider the set
$$
B_{\delta,T}(\varphi)=\{X\in\HH: |\int_0^T \varphi(g_tX)dt|>\delta T\}.
$$

If $\varphi$ is the characteristic function of a specially
chosen large compact set, then J. Athreya \cite{Athreya06}
showed that for any $\delta>0$ the measure
$\mu_{\kappa}(B_{\delta,T}(\varphi))$ decays exponentially
as $T\to\infty$.

Our aim in this paper (see Theorem \ref{LDTteichnew} below)
is to extend the result of Athreya and to establish
exponential decay of $\mu_{\kappa}(B_{\delta,T}(\varphi))$
for a larger class of functions $\varphi$: namely, for
functions, which, following \cite{bufetov06}, we call {\it
  H{\"o}lder in the sense of Veech} (the formal definition
is given in \cite{bufetov06} and repeated below).

It is essential for our proof that $\mu_{\kappa}$ is the
measure of maximal entropy and that the exponential estimate
(\ref{expretest}) holds for $\mu_{\kappa}$.

Unlike Athreya's proof, which relies on the study of the
action of the special linear group on $\HH$, our argument
only uses the symbolic coding for the Teichm{\"u}ller flow
on $\HH$, or, more precisely, for its finite cover --- the
Teichm{\"u}ller flow on Veech's space of zippered
rectangles.

\subsection{Zippered rectangles.}

Here we briefly recall the construction of the Veech space
of zippered rectangles. We use the notation of
\cite{bufetov06}, \cite{BufGur}.

Let $\pi$ be a permutation of $m$ symbols, which will always
be assumed irreducible in the sense that
$\pi\{1,\dots,k\}=\{1,\dots,k\}$ implies $k=m$. The Rauzy
operations $a$ and $b$ are defined by the formulas
$$
a\pi(j)=\begin{cases}
\pi j,&\text{if $j\leq\pi^{-1}m$,}\\
\pi m,&\text{if $j=\pi^{-1}m+1$,}\\
\pi(j-1),&\text{if $\pi^{-1}m+1<j\le m$;}
\end{cases}
$$
$$
b\pi(j)=\begin{cases}
\pi j,&\text{if $\pi j\leq \pi m$,}\\
\pi j+1,&\text{if $\pi m<\pi j<m$,}\\
\pi m+1,&\text{ if $\pi j=m$.}
\end{cases}
$$

These operations preserve irreducibility. The {\it Rauzy class}
$\mathcal R(\pi)$ is defined as the set of all permutations that can
be obtained from $\pi$ by application of the transformation group
generated by $a$ and $b$. From now on we fix a Rauzy class $\R$ and
assume that it consists of irreducible permutations.

For $i,j=1,\dots,m$, denote by $E_{ij}$ the $m\times m$ matrix whose
$(i,j)th$ entry is $1$, while all others are zeros. Let $E$ be the
identity $m\times m$-matrix. Following Veech \cite{veech82}, introduce
the unimodular matrices
\begin{equation} \label{mat_a}
A(\pi, a)=\sum_{i=1}^{\pi^{-1}m}E_{ii}+E_{m,\pi^{-1}m+1}+
\sum_{i=\pi^{-1}m}^{m-1}E_{i,i+1},
\end{equation}
\begin{equation} \label{mat_b}
A(\pi, b)=E+E_{m,\pi^{-1}m}.
\end{equation}

Let
$$
\Delta_{m-1}=\{\la\in {\mathbb R}^m:|\la|=1,\ \la_i>0 \text{ for }
i=1,\dots,m\}.
$$
Denote
$$
\Delta_{\pi}^+=\{\la\in\Delta_{m-1}| \ \la_{\pi^{-1}m}>\la_m\},\ \
\Delta_{\pi}^-=\{\la\in\Delta_{m-1}| \ \la_m>\la_{\pi^{-1}m}\},
$$

Let $\R$ be a Rauzy class of irreducible permutations.  A
{\it zippered rectangle} associated to the Rauzy class $\R$
is a triple $(\la, \pi, \delta)$, where $\la \in {\mathbb
  R}_+^m, \delta\in{\mathbb R}^m, \pi\in{\cal R}$, and the
vector $\delta$ satisfies the following inequalities:
\begin{equation}
\label{deltaone}
\delta_1+\dots +\delta_i\leq 0,\ \  i=1, \dots, m-1.
\end{equation}

\begin{equation}
\label{deltatwo}
\delta_{\pi^{-1}1}+\dots+\delta_{\pi^{-1}i}\geq 0, \ \  i=1, \dots, m-1.
\end{equation}

The set of all $\delta$ satisfying the above inequalities is a cone in
${\mathbb R}^m$; we shall denote this cone by $K(\pi)$.

The area of a zippered rectangle is given by the expression
$$
Area(\la,\pi,\delta)=
\sum_{r=1}^m \la_rh_r=
\sum_{r=1}^m\la_r(-\sum_{i=1}^{r-1} \delta_i+
\sum_{l=1}^{\pi(r)-1}\delta_{\pi^{-1}l})=
$$
\begin{equation}
\label{ziparea}
\sum_{i=1}^m \delta_i (-\sum_{r=i+1}^m \la_r+\sum_{r=\pi(i)+1}^{m}
\la_{\pi^{-1}r}).
\end{equation}
(again,  our convention is that $\sum_{i=m+1}^m (...)=0$ and $\sum_{i=1}^0
(...)=0$).

Consider the set
$$
{\cal V}(\R)=\{(\la,\pi,\delta):\pi\in\R, \la\in{\mathbb R}^m_+, \delta\in
K(\pi)\}.
$$

In other words, ${\cal V}(\R)$
is the space of all possible zippered rectangles  corresponding to the
Rauzy
class $\R$.

The Teichm{\"u}ller flow  $P^t$ acts on ${\cal V}(\R)$ by the
formula
$$
P^t(\la,\pi,\delta)=(e^{t}\la, \pi, e^{-t}\delta).
$$

Veech also introduces a map $\U$ acting on ${\cal V}(\R)$ by
the formula
$$
{\U}(\la,\pi, \delta)=\begin{cases}
(A(\pi, b)^{-1}\la, b\pi,
A(\pi, b)^{-1}\delta),
&\text{if $\la\in\Delta_{\pi}^+$;}\\
(A(\pi, a)^{-1}\la), a\pi,
A(\pi, a)^{-1}\delta,
&\text{if $\la\in\Delta_{\pi}^-$.}
\end{cases}
$$

The map $\U$ and the flow $P^t$ commute (\cite{veech82}).

The volume form $Vol=d\la_1\dots d\la_md\delta_1\dots d\delta_m$ on ${\cal
V}(\R)$
is preserved under the action of the flow $P^t$ and of the map $\U$.
Now consider the subset
$$
{\mathcal V}^{(1)}(\R)=\{(\la,\pi,\delta): Area(\la,\pi,\delta)=1\},
$$
i.e., the subset of zippered rectangles of area $1$; observe that both
$P^t$
and $\U$
preserve the area of a zippered rectangle and therefore the set ${\cal
V}^{(1)}(\R)$ is invariant under $P^t$ and  $\U$.

Denote
$$
\tau(\lambda, \pi)=(\log(|\la|-\min(\la_m,\la_{\pi^{-1}m})),
$$
and for $x\in {\cal V}(\R)$, $x=(\la, \delta ,\pi)$, write $$
\tau(x)=\tau(\la,\pi).
$$
Now define $$
\Y(\R)=\{x\in{\mathcal V}(\R):|\la|=1\}.
$$
and $$
{\mathcal V}^{(1)}_0(\R)=\bigcup_{x\in\cY(\R), 0\leq t\leq \tau(x)}P^tx.
$$

The set ${\mathcal V}^{(1)}_0(\R)$ is a fundamental domain for $\U$ and,
identifying the points $x$ and $\U x$ in ${\mathcal V}^{(1)}_0(\R)$, we obtain
a
natural flow, also denoted by $P^t$, on ${\mathcal V}^{(1)}_0(\R)$.

The restriction of the measure given by the volume form $Vol$ onto the set
${\mathcal V}^{(1)}_0(\cR)$
will be denoted by $\mu_{\cR}$.
  By a theorem, proven independently and simultaneously
by W.Veech \cite{veech82} and H. Masur \cite{Masur82},  $\mu_{\R}({\mathcal V}^{(1)}_0(\R))<\infty$, and we
shall
in what follows assume that $\mu_{\R}$ is normalized to have total mass
$1$.

Now introduce the vectors $h$ and $a$ by the formulas:
\begin{equation}
\label{hh}
h_{r}=-\sum_{i=1}^{r-1} \delta_i+\sum_{l=1}^{\pi(r)-1}\delta_{\pi^{-1}l}.
\end{equation}
\begin{equation}
\label{aa}
a_i=-\delta_1-\dots -\delta_{i-1}.
\end{equation}

The data $(\la, h, a, \pi)$ determine the zippered rectangle
$(\la,\pi,\delta)$ uniquely.
We now metrize the space of zippered rectangles as follows.

 Take two zippered rectangles $x=(\la,h,a, \pi)$ and $x^{\prime}=(\lap,
h^{\prime}, a^{\prime}, \pip)$.
Write
$$
d((\la,h,a), (\lap, h^{\prime}, a^{\prime}))=
\log \frac{\max_i\frac{\la_i}{\lap_i},\frac {h_i}{h^{\prime}_i},
\frac{|a_i|}{|a^{\prime}_i|},
\frac{|h_i-a_i|}{|h^{\prime}_i-a^{\prime}_i|}}
{\min_i\frac{\la_i}{\lap_i},\frac {h_i}{h^{\prime}_i},
\frac{|a_i|}{|a^{\prime}_i|},
\frac{|h_i-a_i|}{|h^{\prime}_i-a^{\prime}_i|}}.
$$
and define the metric on $\Omega(\R)$ by
$$
d(x,x^{\prime})=\begin{cases}
d((\la,h, a), (\lap, h^{\prime}, a^{\prime}))
&\text{if $\pi=\pip$ and $\frac {a_m}{a^{\prime}_m}>0$;}\\
2+d((\la,h,a), (\lap, h^{\prime}, a^{\prime})), &\text{otherwise}.
\end{cases}
$$

We say that a function $f$ on the space of zippered rectangles is {\it
H{\"o}lder} if it is H{\"o}lder with respect to the Hilbert metric introduced
above.

\subsection{Zippered rectangles and abelian differentials.}

Veech \cite{veech82} established the following connection
between zippered rectangles and moduli of abelian differentials.
A detailed description of this connection is given in
\cite{KontZoric03}.

A zippered rectangle naturally defines a Riemann surface
endowed with a holomorphic differential. This correspondence preserves
area.
The orders of the singularities of $\omega$ are uniquely
defined  by the Rauzy class of the permutation $\pi$ (\cite{veech82}).
For any $\R$ we thus have a map
$$
\pi_{\R}: \VR\rightarrow\modk,
$$
where ${\kappa}$ is uniquely defined by $\R$.

Veech \cite{veech82} proved

\begin{theorem}[Veech]
\label{zipmodule}
\begin{enumerate}
\item Up to a set of measure zero, $\pi_{\R}(\VR)$ is a
  connected component of $\modk$.  Any connected component
  of any $\modk$ has the form $\pi_{\R}(\VR)$ for some
  $\R$.
\item The map $\pi_{\R}$ is finite-to-one and almost everywhere locally
bijective.
\item $\pi_{\R}(\U x)=\pi_{\R}(x)$.
\item The flow $P^t$ on $\VR$ projects under $\pi_{\R}$
to the Teichm{\"u}ller flow $g_t$ on the corresponding connected
component of $\modk$.
\item $(\pi_{\R})_*\mu_{\kappa}=\mu_{\R}$.
\item $m=2g-1+\sigma$.
\end{enumerate}
\end{theorem}

A function $\varphi$ on $\modk$ is called {\it Hoelder in the
  sense of Veech} if if there exists a H{\"o}lder function
$\theta:\VR\to {\mathbb R}$ such that
$\varphi\circ\pi_{\cR}=\theta$. In particular if a function
$\varphi: \HH\to{\mathbb R}$ is a lift of a smooth function
from the underlying moduli space ${\mathcal M}_g$ of compact
surfaces of genus $g$, then $\varphi$ is H{\"o}lder in the
sense of Veech (see Remark 3 on p.587 in \cite{bufetov06}).

The main result of this paper is

\begin{maintheorem}
  \label{LDTteichnew}
  Let $\HH$ be a connected component of the moduli space
  $\modk$ of abelian differentials with prescribed
  singularities, let $g_t$ be the Teichm{\"u}ller flow, and
  let $\mu_{\kappa}$ be the smooth measure.  Let
  $\varphi:\HH\to{\mathbb R}$ be bounded and H\"older in the
  sense of Veech. If $\mu_\kappa(\varphi)=0$ and
  $\int_0^\tau\vfi(g_tz)\,dt\neq0$ for some periodic point
  $z$ with period $\tau>0$, then for any $\varepsilon>0$ the
  limit superior
  \begin{align*}
    &\limsup_{T\to+\infty}\frac1T\log\mu_{\kappa}\Big\{x\in \HH:
    \big|\int_0^T\varphi(g_tx)dt\big|\ge
    T\epsilon\Big\}
  \end{align*}
is strictly negative.
\end{maintheorem}

\subsection{Symbolic coding for the Teichm\"uller flow.}

The Teichm{\"u}ller flow on Veech's space of zippered
rectangles admits a representation as a
suspension flow over the natural extension of the
Rauzy-Veech-Zorich induction map on Veech's space of
zippered rectangles \cite{veech82, veech86, zorich96}.
The Rauzy-Veech-Zorich induction has a natural
symbolic coding, and the Teichm{\"u}ller flow can thus be represented as
a suspension flow over a topological Markov chain with a countable alphabet.
The roof function in this representation depends only on the past;
on the other hand, it is neither H{\"o}lder nor bounded away from zero or infinity.

It is therefore convenient to modify the coding by considering
first returns of the Teichm{\"u}ller flow to an appropriately
chosen subset. It turns out that the induced symbolic representation has much
nicer properties; the method goes back to Veech's 1982 paper \cite{veech82}.

There is a certain freedom in choosing the subset for inducing, and thus
we obtain a countable family of symbolic flows over the countable full shift
which code the Teichm{\"u}ller flow and
whose roof functions are H{\"o}lder  and bounded away from zero; for any Teichm{\"u}ller-invariant probability measure at
least one of them codes a set of probability $1$.

We summarize these facts in the following Proposition, essentially due to Veech \cite{veech82};
a detailed exposition of the proof may be found in \cite{BufGur}.

Let $X={\mathbb Z}^{\mathbb Z}$ be the space of all
bi-infinite sequences over a countable alphabet, and let
$\sigma:X\to X$ be the full right shift. The
H{\"o}lder structure on $X$ is chosen in the usual way:
we say that a function $\varphi: X \to {\mathbb R}_+$ is
H{\"o}lder if there exists a non-negative $\alpha<1$ such
that if sequences $\omega, {\tilde \omega}\in X$
coincide at all indices not exceeding $N$ in absolute value,
then
$$
\left|\varphi(\omega)-\varphi({\tilde\omega})\right|\leq C\alpha^N.
$$

If a function $r:X\to {\mathbb R}_+$ is bounded away from
zero, then we denote by $f_t^r$ the suspension flow over
$\sigma$ with roof function $r$ (or just $f_t$ when the roof
function is clear from the context); by $X_r$ the phase
space of the flow $f_t^r$.

Given a bounded measurable function $\varphi$ on $X_{r}$, we
define a function $\varphi_r$ on $X$ by the formula
\begin{equation}
\varphi_r (\omega):=\int_0^{r(\omega)} \varphi(\omega,t)dt.
\end{equation}

We have then the following Proposition (see \cite{BufGur} and
\cite{bufetov06}):
\begin{proposition}
\label{pr:codingmaps}
  Let $\R$ be a Rauzy class of irreducible permutations.
  There exists a countable family of H{\"o}lder functions
  $r_n$, $n\in {\mathbb N}$, bounded away from zero and such
  that the following holds.  For any $n$ there exists an
  injective map ${\bf i}_n: X_{r_n}\to \VR$ such that
\begin{enumerate}
\item the diagram
$$
\begin{CD}
X_{r_n}@ >{\bf i}_n>> \VR \\
@VV f_t^{r_n}V           @VV P^tV
   \\
X_{r_n}@ >{\bf i}_n>> \VR \\
\end{CD}
$$
 is commutative;

 \item for the Masur-Veech smooth measure $\mu_{\R}$ and all $n$ we have
 $$\mu_{\R}({\bf i}_n(X_{r_n}))=1;$$
 furthermore, the measure $({\bf i}_n)^{-1}_*\mu_{\R}$ is
 the unique measure of maximal entropy for the flow
 $f_t^{r_n}$ on $X_{r_n}$;

\item  for any $P^t$-invariant probability measure $\mu$ on $\VR$,
there exists $n$ such that $\mu({\bf i}_n(X_{r_n}))=1$.

\item if a function $\psi: \VR\to {\mathbb R}$ is H{\"o}lder in the
sense of Veech, then the function
$
\big(\psi\circ {\bf i}_n\big)_r
$
is H{\"o}lder on $\Omega$.
\end{enumerate}
\end{proposition}

This Proposition reduces the problem of large deviations for
the Teichm{\"u}ller flow to that of large deviations for
suspension flows over the full countable shift.  We now
proceed to a study of such suspension flows. Our approach is
based on the work of the first author in~\cite{araujo2006a}
which is an adaptation of the work of Young~\cite{Yo90}.

\subsection{Suspension Flows over the Countable Shift.}
\label{sec:suspens-flows-over}


In what follows we present the notation for symbolic
dynamics found in the papers by Buzzi and Sarig
\cite{Sarig99,buzzi-sarig2003} (see also the survey of
Gurevich and Savchenko \cite{GurevSavch98}) which we use in this
text.

Let $\sigma:X\to X$ be the  shift on the space
$X$ of bi-infinite words on a infinite countable alphabet.
Denote by $\M_\sigma$ the family of all $\sigma$-invariant
Borel probability measures on $X$.

We write $[x]_n$ to denote the cylinder of points in $X$
with the same coordinates as $x$ in the positions
$0,\pm1,\dots,\pm(n-1)$, i.e.
\begin{align*}
[x]_n:=\{y\in X: y_i=x_i, i\in\ZZ, |i|<n\}.
\end{align*}
We say that a function $\vfi:X\to\RR$ is \emph{
$(A,\alpha)$-H\"older-continuous} if  $A>0, 0<\alpha<1$
are such that $\var_k(\vfi)\le A\alpha^k$ for all $k\ge1$,
where
\begin{align*}
  \var_k(\vfi)=\sup\{|\vfi(x)-\vfi(y)|:x,y\in X, y\in[x]_k\}.
\end{align*}
We also use the notion of \emph{summable
  variation}: a function $\vfi:X\to\RR$ is of summable
variation if $\sum_{k\ge1}\var_k(\vfi)<\infty$.

We say that a $\vfi:X\to\RR$ is
\emph{$\log$-H{\"o}lder} if there exist  $C,
\alpha>0$ such that  for all $k\in
{\NN}$ and $x\in X$
$$
1-Ce^{-\alpha k}
\leq
\dfrac{\vfi(y)}{\vfi(x)}
\leq
1+C e^{-\alpha k}
\quad\text{for all  } y\in X \text{   with   } y\in[x]_k.
$$
We note that any of these conditions allows $\vfi$ to be
unbounded and implies the continuity of $\vfi$. For H\"older
and summable variation we get uniform continuity.  Moreover,
denoting
\begin{align*}
  \var_k(\vfi,x)=\sup\{|\vfi(x)-\vfi(y)|:y\in X, x_i=y_i
  \text{ for all } |i|<k\}
\end{align*}
we see that $\var_k(\vfi,x)\le\var_k(\vfi)$ for all $k\ge1$
if $\vfi$ is of summable variation, and that for a
$\log$-H\"older $\vfi$ we get $\var_k(\vfi,x)\le C
e^{-\alpha k} |\vfi(x)|$, which now depends on
$\vfi(x)$. Hence a $\log$-H\"older observable never has
summable variation, unless $\vfi$ is bounded. In fact, it is
easy to see that
\begin{lemma}
  \label{le:sumlogHolder}
  If $\vfi$ is H\"older, then $\vfi$ is of summable
  variation.  If $\vfi$ is bounded and $\log$-H\"older, then
  $\vfi$ is H\"older.
\end{lemma}

We also say that an observable $\vfi:X\to\RR$ is
\emph{cohomologous to the zero function} if there exists a
uniformly continuous function $\chi:X\to\RR$ such that
$\vfi=\chi\circ\sigma-\chi$.

We use the following standard notation for Birkhoff sums of
a function $\vfi:X\to\RR$ with respect to a transformation
$f:X\circlearrowleft$ on a space $X$: $S_n^f\vfi :=
\sum_{i=0}^{k-1}\vfi\circ f^i$. We just write $S_k\vfi$ if
the dynamics is clear from the context.

We recall that a Gibbs equilibrium state with respect to a
potential $\psi:X\to\RR$ is, according to Bowen \cite{Bo75}
and Sarig~\cite{Sarig99}, a probability measure
$\mu=\mu_\psi$ on $X$ such that there exists
$P=P_\mu(\psi)\in\RR$ and $K=K_\psi>0$ satisfying
\begin{align*}
  \frac1K
  \le
  \frac{\mu([x]_k) }{e^{-Pk + S_k\psi(x)}}
  \le K,\quad\text{for every}\quad
x\in X\quad\text{and all}\quad k\ge0.
\end{align*}
It is well known that in this case we have
\begin{align}
  \label{eq:equilibrium-state}
  P=\sup_{\nu\in\M_\sigma}\big(h_\nu(\sigma)+\int\psi\,d\nu\big)
  =h_\mu(\sigma)+\int\psi\,d\mu
\end{align}
so that $\mu$ achieves the supremum above.

\begin{maintheorem}
  \label{mthm:deviation-count-shift}
  Let $\sigma:X\to X$ be a countable full shift and
  $\psi:X\to\RR$ be a $\log$-H\"older function.  We assume
  that $\mu=\mu_\psi$ is the \emph{unique} Gibbs equilibrium
  state with respect to $\psi$.  Then for every observable
  $\vfi:X\to\RR$ of summable variation with mean zero
  ($\mu(\vfi)=0$) which is not cohomologous to the zero
  function, we have
  \begin{align*}
    &\limsup_{n\to+\infty}\frac1n\log\mu\{x\in X:
    |S_n\vfi(x)| \ge n\epsilon\}
    \\
    &\le
    \sup\Big\{h_\nu(\sigma)-\int \psi\,d\nu:
    |\nu(\vfi)|\ge\epsilon, \nu\in\M_\sigma, \psi\in
    L^1(\nu)\Big\},
    \text{  and}
    \\
    &\liminf_{n\to+\infty}\frac1n\log\mu\{x\in X:
    |S_n\vfi(x)| > n\epsilon\}
    \\
    &\ge
    \sup\Big\{h_\nu(\sigma)-\int \psi\,d\nu:
    |\nu(\vfi)|>\epsilon, \nu\in\M_\sigma, \psi\in
    L^1(\nu)\Big\}
  \end{align*}
  for every $\epsilon>0$. In addition the supremum above is
  strictly negative.
\end{maintheorem}

Based on this result we are able to obtain the following
large deviation law for a suspension flow over a full
countable shift with respect to the measure naturally
induced by the Gibbs measure in the setting of
Theorem~\ref{mthm:deviation-count-shift}.

Let $r:X\to[r_0,+\infty)$ be a $\log$-H\"older roof function
with $r_0>0$ a constant, and denote by $X_r$ the space
\begin{align*}
  \big\{ (x,t) \in X\times[0,+\infty) :  0\le t < r(x)\big\}.
\end{align*}
Let $f_t:X_r\to X_r, t\ge0$ be the special flow over the
shift $\sigma$ with roof function $r$ (see
e.g. \cite{CoFoSi82}).

We say that an observable $\vfi:X\to\RR$ has
\emph{exponential tail} if there exist
$\epsilon_0>0$ such that $\int
e^{\epsilon_0|\vfi|}\,d\mu<\infty$.


It is well known that given a $\sigma$-invariant probability
$\mu$ there exists a naturally induced $f_t$-invariant
measure $\mu_r$ on $X_r$ (see e.g.\cite{CoFoSi82}).

\begin{maintheorem}
  \label{mthm:devsemiflow}
  Let $\sigma:X\to X$ be a countable full shift and
  $r:X\to[r_0,+\infty)$ be a $\log$-H\"older function with
  exponential tail and $r_0>0$.  We assume that $\mu$ is the
  unique Gibbs equilibrium state with respect to $\psi=-h
  \cdot r$ for some fixed constant $h>0$, and let
  $f_t:X_r\to X_r$ be the flow under $r$ with induced
  $f_t$-invariant measure $\mu_r$. For every \emph{bounded}
  observable $\vfi:X_r\to\RR$ with mean zero
  (i.e. $\mu_r(\vfi)=0$) we denote
  $\vfi_r(x):=\int_0^{r(x)}\vfi\big(f_t(x,0)\big) \, dt$ for
  $x\in X$ and assume that
  \begin{itemize}
  \item $\vfi_r:X\to\RR$ is H\"older, and
  \item there exists a periodic point $z=f_\tau(z)$ with
    some period $\tau>0$, such that $\int_0^\tau
    \vfi(f_t(z))\, dt\neq0$.
  \end{itemize}
  Then we have, denoting for simplicity $\ov{r}=\mu(r)$
  \begin{align*}
    &\limsup_{T\to+\infty}\frac1T\log\mu_r\Big\{z\in X_r:
    \big|\int_0^T
    \hspace{-0.2cm}\vfi\big(f_t(z)\big)\,dt\big|
    \ge \epsilon T\Big\}
    \\
    &\le
    \sup\Big\{h_\nu(\sigma)-\int \psi\,d\nu:
  |\nu(\vfi_r)|\ge\epsilon\overline{r} ,
  \nu\in\M_\sigma, \psi\in
  L^1(\nu)\Big\}.
  \end{align*}

 In addition the supremum above is strictly negative.

 Moreover, in the same conditions above if, in addition, the
 observable $\vfi$ has \emph{compact support}, then we have
  \begin{align*}
  &\liminf_{T\to+\infty}\frac1T\log
  \mu_r\big\{z\in X_r:
  \Big|\int_0^T \vfi\big(f_t(x,0)\big) \,
  dt\Big|\le \epsilon T  \big\}
  \\
  &\ge
  \frac1{r_0}\sup\big\{ h_\nu(\sigma)-\int\psi \,d\nu:
  |\nu(\vfi_r)|>
  \frac{\epsilon\overline{r}}{r_0},
  \nu\in \M_\sigma,
  \psi\in L^1(\nu)\big\}.
  \end{align*}
\end{maintheorem}

The fact that the lower bound for the rate in
Theorem~\ref{mthm:devsemiflow} is different from the upper
bound seems to be a limitation of the method of proof. The
authors believe an adaptation of the methods of Waddington
\cite{wadd96} to this setting should provide sharper results.


\subsection{Organization of the paper}
\label{sec:organiz-paper}

In the next Section~\ref{sec:large-deviat-gibbs} we prove
Theorem~\ref{mthm:deviation-count-shift} adapting the
arguments from Young in \cite{Yo90} to a full countable
shift.  In Section~\ref{sec:large-deviat-maximal} we prove
Theorem~\ref{mthm:devsemiflow} after reducing the estimates
of large deviation for the semiflow to estimates of certain
sets of deviations for adequate observables on the base
transformation, to which we apply
Theorem~\ref{mthm:deviation-count-shift}. Finally, in the
last Section~\ref{sec:exampl-applic} we use Theorem~\ref{mthm:devsemiflow} 
to complete the proof of Theorem~\ref{LDTteichnew}.

\subsection*{Acknowledgements}

We are deeply grateful to Boris M.Gurevich, Amir Dembo and Dmitry Dolgopyat
for  useful discussions.  Part of this work was done while
while V.A. was visiting Rice University, and another part
while A.I.B. was visiting the IMPA. We would like to thank the respective
host institutions for their warm hospitality.  V.A. was partially
supported by CNPq, FAPERJ and PRONEX (Brazil).
A.I.B. is supported in part by the National
Science Foundation under grant DMS 0604386, by the Edgar
Odell Lovett Fund at Rice University and by the Programme on Mathematical Control Theory 
of the Presidium of the Russian Academy of Sciences.



\section{Large deviations for a Gibbs measure on the full
  countable shift}
\label{sec:large-deviat-gibbs}

Here the dynamics is given by $\sigma:X\to X$, the full
countable shift.  We assume that $\vfi:X\to\RR$ if of
summable variation, $\psi$ is $\log$-H\"older with
exponential tail (which ensures that $\psi\in L^1(\mu)$ in
particular).  Without loss of generality, we assume also
that $\mu(\vfi)=0$ and $\vfi\not\equiv0$ in what follows.
For a given $\epsilon>0$ we consider
\begin{align*}
  D_n^\epsilon=\big\{x\in X:  S_n \vfi (x) \ge n\epsilon\big\}.
\end{align*}

The following lemmas are useful tools during the proof.

\begin{lemma}
  \label{le:same-k}
  Let $g:X\to\RR$ be a summable variation function and
  $A_0=\sum_{k\ge1}\var_k(g)$. Suppose $y$ differs from
  $x\in X$ is a single coordinate $0\le |i|<n$. Then
\begin{align*}
  |S_ng(x)-S_ng(y)| \le \sum_{k=0}^{n-1}
  \var_k(g)\le A_0
\end{align*}
Moreover for given $\epsilon>0$ let $n$ be such that
  $\epsilon-A_0/n<\epsilon/2$ and let $x\in X$ be such that
  $|S_n g(x)|> n\epsilon$. For any $y\in X$ with
  $x_i=y_i$ for all $|i|<n$, then $|S_n g(y)|\ge n\epsilon/2$.
\end{lemma}

\begin{proof}
  Just observe that if $x,y\in X$ share the same coordinates
  except the $i$th one with $|i|<n$, then $\sigma^k x,
  \sigma^k y$ share the same coordinates except the
  $(i-k)$th one, thus
  \begin{align*}
    |g(\sigma^k x) - g(\sigma^k y)| \le \var_{|i-k|}(\vfi)
  \end{align*}
  and the first statement follows. For the second just note that
  \begin{align*}
    |S_n g(y)|
    \ge
    |S_ng(x)|-|S_ng(x) + S_ng(y)|
    \ge
    n\epsilon - A_0= n(\epsilon-A_0/n) \ge n\epsilon/2.
  \end{align*}
\end{proof}

From Lemma~\ref{le:same-k} we deduce that, if we fix a
symbol $a$ and define $(\cdot)^a_j: X\to X, x\mapsto x^a_j$
where $x^a_i=x_i$ for $i\neq j$ and $x^a_j=a$,
and also $x^a$ for $x^a_0$, we have
\begin{align*}
  x\in D_n^{\epsilon}\implies x^a\in D_n^{\epsilon/2}.
\end{align*}

\begin{lemma}
  \label{le:alaLivsic}
  Let $\vfi:X\to\RR$ be of summable variation
  (H\"older). Assume that $S_p\vfi(z)=0$ for every
  $\sigma$-periodic point $z$ with period $p\in\NN$. Then
  there exists a uniformly continuous function
  (respectively, H\"older) $\chi:X\to\RR$ so that
  $\vfi=\chi\circ\sigma-\chi$.
\end{lemma}

This lemma says that if a summable variation observable sums
to zero over every periodic orbit, then this observable is
cohomologous to the zero function.

\begin{proof}
  We just follow the usual proof of Livsic's Theorem: since
  $X$ is the full countable shift, let $\omega\in X$ be a
  point with dense positive $\sigma$-orbit and define
  $\chi(\omega):=0$ and
  $\chi(\sigma^n\omega):=\sum_{j=0}^{n-1}\vfi(\sigma^j\omega)$.

  Then, for any $l\in\ZZ^+$, if $x^n=\sigma^n\omega$ and
  $m>n$ satisfy $x^m\in[x^n]_l$, we define
  $z:=\overline{x^n_{0} \dots x^n_{m-n-1}}$ the
  $\sigma$-periodic point with period $m-n$ closest to
  $x^n$, i.e. $z$ is periodic with period $m-n$ and $z\in
  [x]_n$. By construction we have that the $j$th coordinate
  of $x^n$ and $z$ coincide for $j=0,\dots,l+m-n$ and by
  assumption $S_{m-n}\vfi(z)=0$. Thus
  \begin{align*}
   \var_l(\chi)
   &\le
   |\chi(x^m)-\chi(x^n)|
   =
    \left|\sum_{j=n}^{m-n-1}\vfi(x^j)\right|
    =
    \left|\sum_{j=n}^{m-n-1}[\vfi(x^j)-\vfi(\sigma^j
      z)]\right|
    \\
    &\le
    \sum_{j=n}^{m-n-1}|\vfi(x^j)-\vfi(\sigma^jz)|
    \le
    \sum_{j=n}^{m-n-1}\var_{j+l}(\vfi)
    \le
    \sum_{j>l}\var_l(\vfi).
  \end{align*}
  This shows that $\var_l(\chi)\xrightarrow[l\to+\infty]{}
  0$ and so $\chi$ is a uniformly continuous function.  For
  each $n\in\ZZ^+$ it is easy to see that
  $\vfi(x^n)=\chi(x^{n+1})-\chi(x^n)$ and since
  $\{x^n\}_{n\in\ZZ^+}$ is dense in $X$ and $\vfi,\chi$ are
  continuous, we get that $\vfi=\chi\circ\sigma-\chi$ as
  stated.
\end{proof}

Hence from  Lemma~\ref{le:alaLivsic} if we assume that
$\vfi$ is not cohomologous to the zero function, then the
following is true
\begin{itemize}
\item[(G)] there exists a periodic point $z\in X$ such that
  $S_p\vfi(z)>0$ where $p$ is the (minimal) period of
  $z$. Then \emph{there exists $\epsilon_1>0$ such that for
    all $0<\epsilon<\epsilon_1$ and for
    all big enough $n>0$ we have $|S_{n}\vfi(z)|> 2\epsilon n$.}
\end{itemize}
Indeed, there exists a periodic point $z$ with period
$p\in\ZZ^+$ such that $|S_p\vfi(z)|\neq0$ and so we can find
$\epsilon_0>1$ so that $|S_{kp}\vfi(z)|>3\epsilon kp$ for
all $k\in\ZZ^+$ and $0<\epsilon<\epsilon_1$. Therefore, for every
$0\le l<p$ and $k>p\epsilon\max\{|S_i\vfi(z)|:0\le i<p\}$
\begin{align*}
  |S_{kp+l}\vfi(z)|=|S_{kp}\vfi(z)+S_l(z)|\ge
  (kp+l)\big(\frac{3\epsilon kp}{kp+l}
  -\frac{S_l\vfi(z)}{kp+l}\big)
  \ge
  2\epsilon(kp+l),
\end{align*}
proving the (G) property.

The following lemma enable us to choose a good cover for
$D_n^\epsilon$.

\begin{lemma}
  \label{le:goodcover}
  Fix a finite subset $\A_0$ of the alphabet $\A$.
  Given a finite family of functions of summable variation
   $\vfi_1,\dots,\vfi_k:X\to\RR$ and of real
  numbers $\alpha_1,\dots,\alpha_k$, consider
  \begin{align*}
    D=\{x\in X: \vfi_i(x)>\alpha_i, i=1,\dots,k\}
  \end{align*}
  and assume $D$ has positive $\mu$-measure.

  Then there exists a periodic point $z\in D$ and, for any
  given big integer $n>0$, there is an integer $m>n$ and a
  \emph{finite} family $\cC_n$ of $m$-separated points in
  $D$ such that, for
  \begin{align*}
    \A_n=\{a\in\A: \text{$a$ is a letter in the first
      $n$ coordinates of some element $x\in\cC_n$}\}
  \end{align*}
  we get
  \begin{enumerate}
  \item for all $x\in\cC_n$ we have $[x]_m\subset D$;
  \item $\sum_{x\in\cC_n} \mu([x]_m) \ge
    \frac{n-1}{n}\cdot\mu(D)$;
  \item the projection $\pi_{n,m}:X\to\A^{m-n}$ onto the coordinates
    $n,\dots,m-1$ of $\cC_n$ contains only letters from
    $\A_0$, i.e.  $\pi_{n,m}(\cC_n)\subset \A_0^{m-n}$;
  \item $z^a_j\in\cC_n$ for all $0\le j\le n$ and
    $a\in\A_n$.
  \end{enumerate}
\end{lemma}

\begin{remark}
  \label{rmk:z_a}
  The periodic point $z$ from (G) belongs to
  $D^{2\epsilon}_n$ for all sufficiently small $\epsilon>0$
  and big enough $n\in\ZZ^+$.

  In addition, for $n$ such that $2A_0/n<\epsilon$, we have
  that $D_n^\epsilon$ contains $z^a_j$ for every symbol $a$
  in $\A_n$ and for each $0\le j\le n$, by
  Lemma~\ref{le:goodcover}.
\end{remark}

\begin{proof}
  Let $\widetilde\cC_n$ be a maximal $n$-separated set in
  $D$, that is, we choose one point in each non-empty
  intersection $[a_0,a_1,\dots,a_{n-1}]\cap D$ for
  $a_0,\dots,a_{n-1}\in\A$. We observer that this set might
  be infinite and that $\{[x]_n: x\in \widetilde\cC_n\}$
  forms a disjoint open cover of $D$.

  Now we choose a convenient finite
  approximation: let $\cC_n$ be a \emph{finite subset
    of $\widetilde\cC_n$} such that
\begin{align}\label{eq:finteapprox}
  \sum_{x\in \widetilde\cC_n\setminus\cC_n}
  \mu([x]_n)
  &\le
  \frac1n\mu(D).
\end{align}
In this way we obtain that
\begin{align*}
  \mu(D)
  &\le
  \mu\big( \cup_{x\in \widetilde\cC_n\setminus\cC_n} [x]_n \big)
  + \mu\big( \cup_{x\in \cC_n} [x]_n \big)
  \le
  \frac1n\mu(D)
  +
  \mu\big( \cup_{x\in \cC_n} [x]_n \big)
\end{align*}
which implies item (2) of the statement for any $m>n$.

We can at this point add finitely many elements of $D$ to
$\cC_n$ according to our convenience. We first define $\A_n$
as the set of all letters at the first $n$ coordinates of
the points of $\cC_n$. Then we take the periodic point $z\in
D$ given by property (G). Finally we redefine $\cC_n$ to
equal the union $\cC_n\cup\{z^a_j: a\in\A_n, 0\le j\le n\}$.

This keeps the above properties and the new set $\cC_n$
satisfies item (4) of the statement.
Since $\vfi_i$ is of summable variation, for each
$x\in\cC_n$ we can find $y\in X$ and $m=m(x)>n$ such that
\begin{enumerate}
\item[(a)] $y_j=x_j$ for $|j|\le m$, in particular $y\in[x]_n$;
\item[(b)] $|\vfi_i(y)-\vfi_i(x)|\le \sum_{k>m-n}\var_k(\vfi)
  <\alpha_i-\vfi_i(x)$  so  that
  $\vfi_i(y)>\alpha_i$ for all $i=1,\dots,k$, and $y\in D$.
\end{enumerate}
Now let $L_0=\#\A_0$. Since $\cC_n^\epsilon$ is finite we
can consider $m_n=\max\{m(x):x\in\cC_n\}$ and then take an
integer $l_n \ge \log\#\cC_n^\epsilon/\log L_0$.  For
$M_n=m_n+l_n$ replace each $x\in\cC_n^\epsilon$ by $y$
satisfying in addition to (a)-(b) above also
\begin{enumerate}
\item[(c)] $(y_{m_n+1},\dots,y_{m_n+l_n})$ are distinct
  points in $\A_0^{l_n}$.
\end{enumerate}
Observe that this ensures the new elements of $\cC_n$
are still distinct points but can be separated in the
$\ell_n$ coordinates following $m_n$.  Note also that the
choice of $\ell_n$ was made to have "enough room" in
$\ell_n$ coordinates to write $\#\cC_n$ distinct
words in $\A_0$ letters. The proof is complete.
\end{proof}

\subsection{The upper bound}
\label{sec:upper-bound}

Here we give the main step of the proof of the upper bound
for the limit superior in the statement of
Theorem~\ref{mthm:deviation-count-shift}.  From now on we
take $\cC_n$ to be the cover of $D_n^\epsilon$ provided by
Lemma~\ref{le:goodcover}, where we take $i=1$ and
$\alpha_1=n\epsilon-\omega=n(\epsilon-\omega/n)$ for some
small $\omega>0$.  We also set $\hat\psi:=P-\psi$, where
$P=P_\mu(\psi)$ from~\eqref{eq:equilibrium-state}.

\subsubsection{Choose a good sequence of probability
  measures from the covering}
\label{sec:choose-good-sequence}

We consider the families of probability measures
\begin{align*}
  \eta_n&:=\frac1{Z_n}\sum_{x\in\cC_n}e^{- S_n \hat\psi(x)}\cdot
  \delta_x
  \text{ where }
  Z_n:=\sum_{x\in\cC_n} e^{- S_n \hat\psi(x)}   \text{ and}
  \\
  \nu_n
  &:=
  \frac1n\sum_{j=0}^{n-1} \sigma^j_*(\eta_n).
\end{align*}
Note that from the assumption that $\mu$ is a Gibbs
equilibrium measure for $\hat\psi$ we get
\begin{align}\label{eq:Znbound}
  Z_n\le\frac1K \sum_{x\in\cC_n} \mu([x]_n)
    \le \frac1K
\end{align}
since, by the definition of $\cC_n$, the cylinders $[x]_n,
[y]_n$ with distinct $x,y\in\cC_n$ must be disjoint.

\subsubsection{Tightness of the sequence $\nu_n$}
\label{sec:tightn-sequence-nu_n}

The following simple argument shows that we can assume
$\eta_n(\sigma^{-j}[a])>0$ for every letter $a$ in $\A_n$.

\begin{remark}\label{rmk:positive-eta}
  The probability measure $\eta_n$, defined above for the
  set $D_n^{\epsilon}$, satisfies $\eta_n([a])\ge
  e^{-S_n\hat\psi(z^a)}/Z_n>0$ since $D_n^\epsilon$ contains
  $z^a$ for every symbol $a$ in $\A_n$, from
  Remark~\ref{rmk:z_a}. The same argument with $z^a_j$ for
  $0\le j \le n$ in the place of $z^a$ shows that
  $\eta_n(\sigma^{-j}[a])>0$ for every letter $a$ in $\A_n$.
\end{remark}

  \begin{lemma}
    \label{le:zeta}
    Let us define for each letter
    $b$ of $\A_n$ and each $0\le j<n$
    \begin{align*}
      \zeta^n_b(j):=\sum_{x\in\cC_n\cap\sigma^{-j}[b]}
      e^{-S_{n}\hat\psi(x)+\hat\psi(\sigma^j x)}.
    \end{align*}
    There exists a constant $L>0$ such that $\zeta^n_a(j)\ge
    L$ for every $a\in\A_n$, all $n>0$ and each $0\le
    j<n$.
  \end{lemma}

\begin{proof}
  Fix some symbol $a\in\A_n$. For $n\in\ZZ^+$ big enough so
  that property (G) holds and for $0\le j<n$ write


  \begin{align*}
    \zeta^n_a(j)
    &=
    \sum_{b_{0},\dots,\hat{b_j},\dots,b_{n-1}}
    \sum_{\substack{x\in\cC_n \\
        x_{0}=b_{\ell},\dots,x_j=a,\dots,x_{n-1}=b_{n-1}}}
    e^{- S_{j}\hat\psi(x)
      -S_{n-j-1}\hat\psi(\sigma^{j+1}x)}
    \\
    &\ge \sum_{b_{0},\dots,\hat{b_j},\dots,b_{n-1}}
    \hspace{-.5cm}
    K^2\mu([b_{0},\dots,b_{j-1}])\mu([b_{j+1},\dots,b_{n-1}])
    \hspace{-1.2cm}\sum_{\substack{x\in\cC_n \\
        x_{0}=b_{0},\dots,x_j=a,\dots,x_{n-1}=b_{n-1}}}
    \hspace{-1.2cm}e^{-\hat\psi(\sigma^j x)},
  \end{align*}
  where we have used the Gibbs property only and write
  $\hat{b_j}$ to denote the \emph{absence of $b_j$} in the
  index of the sum above. Now using the fact that $z^a_j$
  belongs to $\cC_n\cap\sigma^{-j}[a]$ and that $\hat\psi$
  is $\log$-H\"older, we bound the last summand as follows
  \begin{align*}
    \sum_{\substack{x\in\cC_n \\
        x_{\ell}=b_{\ell},\dots,x_j=a,\dots,x_{n-1}=b_{n-1}}}
    \hspace{-1.2cm}e^{-\hat\psi(\sigma^j x)} \ge
    e^{-\hat\psi(\sigma^j z^a_j)}
    \ge e^{-\hat\psi(\sigma^j z)-var_j(\hat\psi,z)} .
  \end{align*}
  Since this bound does not depend on the choice of
  $b_{\ell},\dots,\hat{b_j},\dots,b_{n-1}$ we conclude that
  $\zeta^n_a(j)\ge K^2 e^{-\hat\psi(\sigma^j z)-var_j(\hat\psi,z)}$. To
  obtain the statement of the lemma we set
  \begin{align*}
    L&=\min\{K^2e^{-\hat\psi(\sigma^j z)-\var_j(\psi,z)}:
    0\le j < n, a\in\A_n\}
    \\
    &=\min\{K^2e^{-\hat\psi(\sigma^j z)-\var_1(\psi,z)}:
    0\le j < p\}
  \end{align*}
  since, for all big enough $n$, the period $p$ of $z$ is
  smaller than $n$, and $\var_j(\psi,z)\le
  C|\psi(z)|e^{-\alpha j}\xrightarrow[j\to+\infty]{}0$. The
  lower bound does not depend either on $n$ or on $\A_n$.
\end{proof}

Consider now the sequence of measures $\nu_n$ and $\eta_n$
defined above for $D_n^\epsilon$.

\begin{proposition}
  \label{pr:tightness}
  There exists a constant $C_2>0$ such that for every symbol
  $a$ in $\A_n$ we have $ \nu_n([a]) \le C_2\mu([a])$ for
  all $n$ sufficiently big.
\end{proposition}

This shows in particular that the sequence $(\nu_n)_{n\ge1}$
is tight.

\begin{proof}

We need the following lemma.

\begin{lemma}
  \label{le:rtight}
  There exists $C_2>0$ such that $\eta_n(\sigma^{-j}[a])\le
  C_2\mu([a])$ for every $n\in\ZZ^+$, each $0\le j<n$ and for
  every symbol $a\in\A_n$.
\end{lemma}

\begin{proof}
  Fix a symbol $a\in\A_n$ and $0\le j<n$. We have
  \begin{align*}
    \eta_n(\sigma^{-j}[a]) &=
    \frac1{Z_n}\sum_{x\in\cC_n\cap\sigma^{-j}[a]}
    e^{-S_{j}\hat\psi( x) - \hat\psi(\sigma^j x) -
      S_{n-j-1}\hat\psi(\sigma^{j+1}x)} \le K\mu([a])\cdot
    \frac{\zeta_a^n(j)}{Z_n}
  \end{align*}
  since $e^{-\hat\psi(\sigma^j x)}\le e^{- \inf
    (\hat\psi\mid [a])}\le K\mu([a])$ by the Gibbs property
  of $\mu$.  We can bound $Z_n$ using Lemma~\ref{le:zeta} as
  follows
  \begin{align*}
    Z_n &= \sum_{b}\sum_{x\in\cC_n\cap[b]}
    e^{-\hat\psi(x)} e^{-S_{n-1}\hat\psi(\sigma x)} \ge \sum_{b}
    \frac{\mu([b])}K \sum_{x\in\cC_n\cap[b]}
    e^{-S_{n-1}\hat\psi(\sigma x)}
    \\
    &\ge \sum_{b} \frac{\mu([b])}K \cdot \zeta^n_b \ge
    \frac{L}K.
  \end{align*}
  Finally we find an upper bound for $\zeta^n_a$ using again
  the Gibbs property of $\mu$
  \begin{align*}
    \zeta^n_a(j) &\le
    \sum_{x\in\cC_n\cap\sigma^{-j}[a]}
    K\mu([x]_{j}) \cdot K \mu([\sigma^{j+1}x]_{n-j-1}) \le
    K^2
  \end{align*}
  since $\cC_n$ is a $n$-separated subset.

  This shows that $\eta_n(\sigma^{-j}[a])\le K\mu([a])\cdot
  K^2 /(L/K) = (K^4/L) \cdot\mu([a])$ and concludes the
  proof.
\end{proof}

Now since the bounds in Lemmas~\ref{le:zeta}
and~\ref{le:rtight} do not depend on $0\le j<n$ for all big
enough $n$, we see that for any given $a\in\A_n$ and
sufficiently big $n$ we have
\begin{align*}
   \frac1n\sum_{j=0}^{n-1}
  \eta_n(\sigma^{-j}[a]) = \nu_n([a]) \le C_2 \mu( [a])
\end{align*}
concluding the proof of Proposition~\ref{pr:tightness}.
\end{proof}

\subsubsection{Upper bound for large deviations on the base dynamics}
\label{sec:upper-bound-large}

Using the definition of $\nu_n$ and $Z_n$ and observing that
for all $n>0$
\begin{align*}
  \nu_n(\vfi)
  &=
  \frac1n\sum_{j=0}^{n-1} \eta_n(\vfi\circ\sigma^j)
  =
  \frac1{Z_n}\sum_{x\in\cC_n} e^{-S_n\hat\psi(x)}\cdot
  \frac1n\sum_{j=0}^{n-1} \vfi(\sigma^jx)
  > \epsilon -\frac{\omega}{n}
\end{align*}
we see that any weak$^*$ accumulation point $\nu$ of the
sequence $\nu_n$ satisfies $\nu(\vfi)\ge\epsilon$.
In what follows we assume without loss of generality that
$\nu_n$ converges to $\nu$ when $n\to\infty$ in the weak$^*$
topology.

On the one hand since $\{[x]_n : x\in\cC_n\}$ is an
approximate cover of $D_n^\epsilon$ from
Lemma~\ref{le:goodcover} and the Gibbs property we have
\begin{align*}
  \limsup_{n\to+\infty}\frac1n\log\mu(D_n^\epsilon)
  &\le
  \limsup_{n\to+\infty}\frac1n\log K\frac{n}{n-1}
  \sum_{x\in\cC_n} e^{-S_n\hat\psi(x)}
  \\
  &=
  \limsup_{n\to+\infty}\frac1n\log Z_n.
\end{align*}
On the other hand, considering the following
partition\footnote{The complicated choice of the covering in
  Lemma~\ref{le:goodcover} was to be able to choose a finite
  partition here.}  of $X$
\begin{align*}
  \cP=\{[a]:a\in \A_0\}\cup\{x\in X: x_0\not\in\A_0\},
\end{align*}
we note that by the choice of the points in
$\cC_n$ the refined partition
\begin{align*}
  \cP^{M_n}:=\bigvee_{|i|< M_n}\sigma^i \cP
\end{align*}
\emph{separates the elements of}
$\cC_n$: there is at most one element of
$\cC_n$ in each atom of $\cP^n$. From \cite[Lemma
9.9]{Wa82} we have
\begin{align*}
  H_{\nu_n}(\cP^{M_n})
  -
  \int S_n \hat\psi(x) \, d\nu_n(x)
  =
  \log \sum_{x\in\cC_n} e^{-S_n \hat\psi(x)}.
\end{align*}
From this we deduce following standard arguments (see e.g.
\cite[pag.  220]{Wa82}) that for every $1<q<n$, denoting by
$\#\cP$ the number of elements of the partition $\cP$
\begin{align}\label{eq:q}
  \frac1{n}\log Z_n
  &\le
  \frac1q H_{\nu_n} (\cP^q ) + \frac{2q}{n}\log\#\cP -
  \int\hat\psi\, d\nu_n.
\end{align}

Now for the final step we need the following.

\begin{lemma}
  \label{le:uniformintegral}
  We have
  $\nu_n(\hat\psi)\to\nu(\hat\psi)$ when $n\to\infty$.
\end{lemma}

\begin{proof}
  Using the $\log$-H\"older property and $\mu$-integrability
  of $\hat\psi$ we get, for any given fixed $x\in X$
  \begin{align*}
    \infty
    &>
    \mu(|\hat\psi|)=\sum_a \mu(|\hat\psi|\chi_{[a]})
    \ge
    \sum_a (|\hat\psi(x^a)|-\var_1(\hat\psi,x^a))\mu([a])
  \end{align*}
  thus $\sum_a
  |\hat\psi(x^a)|\mu([a])\le\mu(|\hat\psi|)+\var_1(\hat\psi,x^a)<\infty$.
  \footnote{The same argument shows in fact that $\psi\in
    L^1(\mu)\iff \sum_{a\in\A} |\psi(x^a)|<\infty$ for any
    given fixed $a\in\A$.}

  Given a function $g:X\to\RR^+$ define for each $L>0$
  the function $g_L$ to equal $g$ if $g>L$ and $0$
  otherwise.

  Now from the $log$-H\"older property of $\hat\psi$ and the
  $\mu$-integrability $\hat\psi$, together with
  Proposition~\ref{pr:tightness}, we obtain for every big
  enough $n$ and for positive $L$
  \begin{align*}
    \nu_n(|\hat\psi|_L)
    &=
    \sum_a \nu_n(|\hat\psi|_L \cdot \chi_{[a]})
    \le
    \sum_{a\,:\,\sup\hat\psi\mid[a]>L}
    (|\hat\psi(x^a)|+\var_1(\hat\psi,x^a))\cdot\nu_n([a])
    \\
    &\le
    C_2 \sum_{a\,:\,\sup\hat\psi\mid[a]>L}
    (|\hat\psi(x^a)|+\var_1(\hat\psi,x^a))\mu([a])
    \\
    &\le
    \int_{|\psi|>L} (|\psi(x)|+2\var_1(\psi,x))\,d\mu(x)
    \\
    &\le
    \int_{|\psi|>L} |\psi(x)|(1+2Ce^{-\alpha})\,d\mu(x)
     \xrightarrow[L\to+\infty]{}0
  \end{align*}
  which shows that $\nu_n(\hat\psi)$ is a uniformly convergent
  sequence of integrals.
\end{proof}

From inequality \eqref{eq:q} and
Lemma~\ref{le:uniformintegral} we conclude
\begin{align}
  \limsup_{n\to+\infty}\frac1n\log Z_n
  &\le
  \frac1q
  \limsup_{n\to+\infty} H_{\nu_n} (\cP^q ) +
  \limsup_{n\to+\infty} \int -\hat\psi\, d\nu_n\nonumber
  \\
  &\le h_\nu(\sigma,\cP) - \int \hat\psi \, d\nu
  \le
  h_\nu(\sigma) - \nu(\hat\psi).\label{eq:upper}
\end{align}
Finally we note that as a consequence of the assumption that
$\mu$ is the unique Gibbs measure associated to $\psi$, we
have for all $\nu\in\M_\sigma\setminus\{\mu\}$
\begin{align*}
  h_\mu(\sigma)-\mu(P-\psi) = 0 > h_\nu(\sigma) -\nu (P-\psi).
\end{align*}
This shows that \eqref{eq:upper} is negative.

\subsection{The lower bound}
\label{sec:lower-bound}

Let $\nu$ be a $\sigma$-invariant probability measure
satisfying $\vfi,\psi\in L^1(\nu)$ and
$|\nu(\vfi)|>\epsilon$, for a fixed small $\epsilon>0$.
We define
\begin{align*}
  \widehat D_n^\epsilon=\big\{x\in X: S_n \vfi (x) >
  n\epsilon\big\}.
\end{align*}

We will find a sequence $\nu_{n_k}$
of invariant measures converging to $\nu$ such that
$\mu(\widehat D_{n_k}^\epsilon)\ge n_k\cdot
\exp\big(h_{\nu_{n_k}}-\nu_{n_k}(\hat\psi)-2\delta\big)$ for
small $\delta>0$ with $h_{\nu_{n_k}}\ge h_\nu
-\delta$ and $\nu_{n_k}(\hat\psi)\le\nu(\hat\psi)+\delta$. Then
\begin{align}\label{eq:liminf}
  \lim_{k\to+\infty}
  \frac1{n_k}\log\mu(\widehat D_{n_k}^\epsilon)\ge h_\nu -
  \int\hat\psi\, d\nu - 4\delta.
\end{align}
Following the ideas in \cite{Yo90} we approximate $\nu$ by a
finite convex combination of $\sigma$-ergodic measures and
then use their ergodicity and a weak form of specification
to build the separated set which will provide the estimates
for $\mu(\\widehat D_n^\epsilon)$.

\subsubsection{Approximating by ergodic measures}
\label{sec:approx-ergodic-measu}

We use the Ergodic Decomposition Theorem \cite{Ro62,Ro67}
for the measure preserving endomorphism $\sigma$ of the
Lebesgue space $(X,\B,\nu)$, where $\B$ is the Borel
$\sigma$-algebra of $X$.

\begin{theorem}
  \label{thm:ergdecomp}
  There exists a smallest $\sigma$-invariant measurable
  partition $\I$ of $X$ except a set of $\nu$-null
  measure. Let $\{\nu_\xi\}_{\xi\in\I}$ be the
  disintegration of $\nu$ into conditional probability
  measures and $\hat\nu$ be the probability measure induced
  in the quotient space $X/\I$.  Then
  \begin{enumerate}
  \item $\nu_\xi$ are $\sigma$-invariant ergodic probability
    measures for $\hat\nu$-a.e. $\xi\in\I$;
  \item for each $n\ge1$ and every $\nu$-integrable
    function $g:X\to\RR^n$
    \begin{enumerate}
    \item $\xi\in\B\mapsto \nu_\xi(g)$ is $\hat\nu$-integrable;
    \item $\nu(g)=\int \nu_\xi(g) \, d\hat\nu(\xi)$;
    \end{enumerate}
  \item $h_\nu(\sigma)=\int h_{\nu_\xi}(\sigma\mid\xi)\,
    d\hat\nu(\xi)$.
  \end{enumerate}
\end{theorem}

Now we use this to build a finite linear convex combination
of ergodic measures which approximates $\nu$.

\begin{lemma}\label{le:approxergodic}
  Define $g:X/\I\to\RR^3$ by
  $g(\xi)=(\nu_\xi(\vfi),\nu_\xi(\hat\psi),h_{\nu_\xi})$
  (which is $\hat\nu$-integrable) and let
  $0<\delta<(\nu(\vfi)-\epsilon)/4$.

  Then there exists a finite linear convex combination
  $\nu_0$ of ergodic measures such that
  $\|\nu(g)-\nu_0(g)\|<\delta$, where $\|\cdot\|$
  denotes the Euclidean norm in $\RR^3$.
\end{lemma}

\begin{proof}
  Choose $\zeta>0$ so that
  $\zeta/(1-\zeta)<\delta/(2+\|\nu(g)\|)$. Let $\Q$ be
  a denumerable partition of $\RR^3$ into cubes whose
  diameter is smaller than $\zeta$. Let also
  $\Q_0\subset\Q$ be the family of such cubes bounded by a
  cube $[-L,L]^3$, where $L>0$ is big enough so that
  \begin{itemize}
  \item $q=\hat\nu\big(g^{-1}(\cup\Q_0)\big)>1-\zeta$;
  \item $s=\sum_{R\in\Q\setminus\Q_0} \big( |\nu_\xi(\vfi)|
    + |\nu_\xi(\hat\psi)|+h_{\nu_\xi} \big) \cdot
    \hat\nu(g^{-1}R)<\zeta$.
  \end{itemize}
  We can now define a probability measure
  \begin{align*}
    \nu_0=\frac1q\sum_{R\in\Q_0} \hat\nu\big(g^{-1}R\big)
    \cdot \nu_{\xi_R},
  \end{align*}
  where $\nu_{\xi_R}$ is an ergodic measure chosen in
  $g^{-1}(R)$ for each $R\in\Q_0$. Hence $\nu_0$ is a
  finite convex linear combination of $\sigma$-ergodic
  measures.

  Analogously we define a tail measure
  \begin{align*}
    \nu_1=\sum_{R\in\Q\setminus\Q_0}
    \hat\nu\big(g^{-1}R\big) \cdot \nu_{\xi_R},
  \end{align*}
  and note that $\|\nu_1(g)\|\le s<\zeta$.

  Now we check that $\nu_0$ is an approximation of $\nu$:
  \begin{align*}
    \|\nu(g)-\nu_0(g)\| &=
    q^{-1}\|q\nu(g)-q\nu_0(g)\|
    \\
    &= q^{-1}\|q\nu(g)-(q\nu_0+\nu_1)g +\nu_1(g)\|
    \\
    &\le \frac1q\|(q-1)\nu(g)\|+\frac1q\|\nu(g)
    -(q\nu_0+\nu_1)g\| +\frac{\|\nu_1(g)\|}{q}
    \\
    &\le
    \frac{1-q}{q}\|\nu(g)\|+\frac{\zeta}{q}+\frac{\zeta}{q}
    \le (2+\|\nu(g)\|)\frac{\zeta}{1-\zeta} \le \delta.
  \end{align*}
  The proof is complete.
\end{proof}

Write $\nu_0=\sum_{i=1}^k a_i\eta_i$, where $a_i>0$,
$\sum_ia_i=1$ and $\eta_i$ are $\sigma$-ergodic probability
measures.

\subsubsection{Build a good cover using ergodicity and a
  form of specification}
\label{sec:build-good-cover}

Now we strongly use the fact that we have approximated $\nu$
by a combination of \emph{ergodic} measures.  As in
Lemma~\ref{le:zeta} let $A_0=\sum_{k\ge1}\var_k(\psi)$.  Let
$N>1$ be such that
\begin{align*}
  \frac{A_0}{N-k}\le\frac{\delta}4
  \quad\text{and}\quad
  \big|\frac1N\eta_i(\vfi) \big| +\frac{\delta}N
  \le\frac{\delta}{8k}.
\end{align*}
In addition, choose $N$ big enough so that for $n>N$ and
each $i=1,\dots,k$, the subset of $X$
\begin{align*}
  Y_n^i=\big\{
  \frac1{[a_i n]}S_{[a_i n]}\hat\psi
  \le\eta_i(\hat\psi)+\delta
  \quad\&\quad
  \frac1{[a_i n]} S_{[a_i n]}\vfi
  \ge\eta_i(\vfi)-\delta\big\},
\end{align*}
satisfies $\eta_i(Y_n^i)>1-\delta$ (where
$[a]=\max\{j\in\ZZ:j\le a\}$ is the biggest integer less or
equal to $a\in\RR$).  Assume also that $N$ is big enough so
that $\var_{[a_i n]}(\vfi)<\delta/k$ for all $i=1,\dots,k$
and $n>N$.

Using a result from Katok\footnote{Stated only for
  homeomorphisms of compact spaces, but the proof does not
  use this assumption!}~\cite[Theorem 1.1]{katok80} we have
that there exists a $[a_i n]$-separated set $E_n^i\subset
Y_n^i$ with at least $\exp\big([a_i
n](h_{\eta_i}-\delta)\big)$-points. Number the elements of
$E_n^i$ as $x_1^i,\dots,x_{m_i}^i$.

Consider a $k$-tuple $(j_1,\dots,j_k)$ with $1\le j_i\le
m_i$ for $i=1,\dots,k$. There corresponds a point
$y=y_{j_1,\dots,j_k}\in X$ (not uniquely defined) so that
its positive orbit shadows the orbit segments
\begin{align*}
  (x_{j_1}^1, \sigma x_{j_1}^1, \dots, \sigma^{[a_1 n]}
  x_{j_1}^1),
  \dots,
  (x_{j_k}^k, \sigma x_{j_k}^k, \dots, \sigma^{[a_k n]}
  x_{j_k}^k).
\end{align*}
Let $\cE$ be the family of points obtained in this manner
and fix $y\in\cE$.  By the summable variation of $\vfi$, for
$m=\sum_i [a_i n]$ and $n_0=\min_i[a_i n]$ we have
\begin{align*}
  \big|
  S_m\vfi(y)-\sum_{i=1}^k S_{[a_i n]}\vfi(x_{j_i}^i)
  \big|
  \le
  \sum_{i=1}^k\var_{[a_i n]}(\vfi)
  \le
  \delta.
\end{align*}
Now we can write because $a_i n-1\le[a_i n]\le a_i n$
\begin{align*}
  \frac1m S_m\vfi(y)
  &\ge
  \frac1m \sum_{i=1}^k S_{[a_i
    n]}\vfi(x_{j_i}^i)-\delta
  \ge
  \sum_{i=1}^k
  \frac{[a_i n]}m (\eta_i(\vfi)-\delta) -\delta
  \\
  &\ge
  \frac1m\sum_{i=1}^k
  a_i n \cdot (\eta_i(\vfi)-\delta)
  -\frac1m\sum_i (\eta_i(\vfi)-\delta)^+
  -\delta,
\end{align*}
since we must take the sign of $\eta_i(\vfi)-\delta$ into
account, where $a^+=\max\{0,a\}$.  Note that by the choice
of $\delta$ in Lemma~\ref{le:approxergodic} and because
$m\le\sum_i a_i n = n$ we have
\begin{align*}
  \sum_{i=1}^k
  a_i n \cdot (\eta_i(\vfi)-\delta)
  = n\cdot \big(\nu_0(\vfi)-\delta\big)
  \ge n\cdot\big(\nu(\vfi)-2\delta\big)
  >0.
\end{align*}
Together with the choice of $N$ we obtain
\begin{align*}
  \frac1m S_m\vfi(y)
  &\ge
  \frac{n}m \cdot\big(\nu(\vfi)-2\delta\big)-\frac{\delta}8 -\delta
  \ge
  \nu(\vfi)-\frac{25}8\delta>\epsilon.
\end{align*}
This means that $y\in \widehat D_m^\epsilon$.

In addition, note that for different choices of the
$k$-tupples we get distinct points $y,y'\in\cE$ which are
$m$-separated, that is $[y]_m\cap[y']_m=\emptyset$ by
construction.

Finally observe that for $w\in [y]_m$ we have, by
Lemma~\ref{le:same-k}
\begin{align*}
  S_m\vfi(w)
  &\ge S_m\vfi(y) - 2A_0
  \ge
  \big(\nu(\vfi)-\frac{25}8\delta-\frac{\delta}4\big)\cdot m -2A_0
  \\
  &\ge
  \big( \nu(\vfi)-\frac{27\delta}8 -\frac{2A_0}m\big)\cdot m
  \ge
  \big( \nu(\vfi) - \frac{31}8 \delta \big)\cdot m
  > m\cdot\epsilon,
\end{align*}
where we have used that $m=\sum_{i=1}^k[a_i n]\ge
\sum_{i=1}^k( a_i n -1) =n-k$.  Thus $\{[y]_m\}_{y\in\cE}$
is a family of $m$-separated subsets inside $D^\epsilon_m$.

\subsubsection{Estimating the measure of $\widehat D_m^\epsilon$}
\label{sec:estimat-measure}

Finally by the previous arguments we can bound the measure
of $\widehat D_m^\epsilon$ from below. Since $\cE\subset \cup_i Y_n^i$
and $\mu$ is Gibbs
\begin{align*}
  \mu(\widehat D_m^\epsilon)
  &\ge
  \sum_{y\in\cE} \mu\big([y]_m\big)
  \ge
  \sum_{y\in\cE} \frac1K\cdot e^{-S_m\hat\psi(y)}
  \\
  &\ge
  \frac1K
  \sum_{y\in\cE} \exp\big(-\sum_i [a_i n]\cdot(\eta_i(\hat\psi)+\delta)\big).
\end{align*}
We also know that $\#E_n^i\ge\exp\big([a_i
n](h_{\eta_i}-\delta)\big)$ and from this we get
\begin{align*}
  \mu(\widehat D_m^\epsilon)
  &\ge
  \frac1K\cdot \exp \big( \sum_i [a_i n]\cdot(h_{\eta_i}
  - \eta_i(\hat\psi) -2\delta)\big).
\end{align*}
Hence for any given $\delta>0$ there exists a big $N$ so
that for all $n>N$ we can find $m\ge n-k$ satisfying
\begin{align*}
  \frac1m\log\mu(\widehat D_m^\epsilon)
  &\ge
  -\frac1m\log K + \frac1m \sum_{i=1}^k [a_i n]\cdot(h_{\eta_i}
  - \eta_i(\hat\psi) -2\delta).
\end{align*}
By the upper bound on large deviations already obtained, we
know that $h_{\eta_i} - \eta_i(\hat\psi) -2\delta \le 0$ and
hence
\begin{align*}
  \frac1m\log\mu(\widehat D_m^\epsilon)
  &\ge
  -\frac1m\log K + \frac{n}m \sum_{i=1}^k a_i\cdot(h_{\eta_i}
  - \eta_i(\hat\psi) -2\delta)
  \\
  &\ge
  -\frac1m\log K + (h_\nu-\delta) - (\nu(\vfi) +\delta) -2\delta.
\end{align*}
This completes the proof of \eqref{eq:liminf}.

\subsection{The rates}
\label{sec:exact-rate}

Now we obtain explicit expressions for the rates of decay
of the measure of the deviation set. On the one hand, in
Section~\ref{sec:upper-bound} we showed that there exists a
$\sigma$-invariant probability $\nu$ such that
$|\nu(\vfi)|\ge\epsilon$, $\psi$ is $\nu$-integrable and
inequality~\eqref{eq:upper} is true, i.e.
\begin{align}\label{eq:upper1}
  \limsup_{n\to+\infty}\frac1n\log\mu(D_n^\epsilon)
  \le
  h_\nu(\sigma)-\int\hat\psi\,d\nu<0.
\end{align}
On the other hand, in Section~\ref{sec:lower-bound} it was
proved that for every given $\sigma$-invariant probability
$\nu$ such that $|\nu(\vfi)|>\epsilon$, $\psi$ is
$\nu$-integrable, and given $\delta>0$ there exists a
sequence $n_k$ tending to $+\infty$ such
that~\eqref{eq:liminf} is true, that is
\begin{align}
  \label{eq:lower}
  \liminf_{n\to+\infty}\frac1n\log\mu(\widehat D_n^\epsilon)
  \ge
  \sup_{\nu\in\M_\sigma}\big\{
  h_\nu(\sigma)-\!\!\int\!\!\hat\psi\,d\nu :
  |\nu(\vfi)|>\epsilon, \nu(\hat\psi)<\infty\big\}.
\end{align}
From~\eqref{eq:upper1} and~\eqref{eq:lower} we deduce that
the supremo above is also an upper bound for the limit
superior and it is strictly negative. This completes the
proof of Theorem~\ref{mthm:deviation-count-shift}.


\section[Large deviations on special flows]{Large deviations for
  maximal entropy measures for special flows over a full
  countable shift}
\label{sec:large-deviat-maximal}

Here we prove Theorem~\ref{mthm:devsemiflow}.  We assume
that $\mu$ is a $\sigma$-ergodic probability on the full
countable shift $X$ which is a Gibbs measure and the unique
equilibrium state with respect to $\psi=-h\cdot r$, where
$h$ is the topological entropy of the flow
$f_t:X_r\circlearrowleft$ built over $\sigma$ with roof
function $r:X\to[r_0,+\infty)$, with some $r_0>0$.  In
particular $r$ (and $\psi$) is $\mu$-integrable.

This means that the induced $f_t$-invariant probability
measure $\mu_r$ on $X_r$ is the measure of maximal entropy
of the flow.

We assume further that $r$ is $\log$-H\"older with
exponential tail.

\subsection{Reduction to the base dynamics}
\label{sec:reduct-base-dynamics}

Here we describe how to pass from the deviation set for the
suspension flow with respect to a bounded observable with
summable variation, to another deviation set for the base
dynamics with respect to another observable, now unbounded.

Consider a continuous observable $\vfi:X_r\to\RR$ and note
that we may write the time average of $\vfi$ under the
action of the semiflow on the point $z=(x,s)\in X_r$ as
\begin{align*}
  \int_0^T
\vfi\big(f_t(z)\big)\,dt
&=
\sum_{j=1}^{n-1} \int_0^{r(\sigma^j(x))}
\vfi\big(f_t(\sigma^j(x),0)\big)\,dt
+
\int_s^{r(x)}
\vfi\big(f_t(x,0)\big)\,dt
\\
&\
+\int_0^{T+s-S_{n}r(x)}
\vfi\big(f_t(\sigma^n(x),0)\big)\,dt,
\end{align*}
where $n=n(x,s,T)\in\NN$ is such that $S_{n}^{\sigma}r(x)\le s+T <
S_{n+1}^\sigma r(x)$.

Recalling that $\vfi_r(x):=\int_0^{r(x)}\vfi\big(f_t(x,0)\big)\,dt$
for $x\in X$ we obtain
\begin{align}\label{eq:timeTvsn}
  \int_0^T \hspace{-0.2cm}\vfi\big(f_t(z)\big)\,dt
  &=
   S_n^\sigma\vfi_r(x) + I_T(x,s),
\end{align}
where
\begin{align*}
  I_T(x,s)=\int_0^{T+s-S_n r(x)}
  \hspace{-0.7cm}\vfi\big(f_t(\sigma^n(x),0)\big)\,dt -
  \hspace{-0.1cm}
  \int_0^s\hspace{-0.2cm}\vfi\big(f_t(x,0)\big)\,dt.
\end{align*}
Assume now that $\vfi:X_r\to\RR$ is bounded and that
$\vfi_r:X\to\RR$ is H\"older.

Note that $\vfi_r$ is not necessarily bounded. Recall also
that $\mu_r(\vfi)=\mu(\vfi_r)/\mu(r)$. We assume without
loss of generality that $\mu(\vfi_r)=0$. Moreover we also
assume that there exists some $\sigma$-periodic point $z\in
X$, with period $p\in\ZZ^+$, such that
\begin{align}\label{eq:periodicond}
  S_p^\sigma\vfi_r(z)=\int_0^\tau\vfi(f_t(z,0))\,dt\neq0
  \quad\text{where}\quad \tau:=S_pr(z).
\end{align}

\subsection{The limit superior}
\label{sec:limit-superi}

From now on all Birkhoff sums are taken with respect to
$\sigma$.  The previous discussion showed that for $\epsilon>0$
\begin{align*}
  \Big\{z\in X_r: \big|\int_0^T
\hspace{-0.2cm}\vfi\big(f_t(z)\big)\,dt\big| \ge \epsilon
T\Big\}
=
 \Big\{ (x,s)\in X_r :
 \big| S_n\vfi_r(x) + I_T(x,s)\big|\ge \epsilon T \Big\},
\end{align*}
where $n=n(x,s,T)$ as before.  Hence because
\begin{align*}
  \big| S_n\vfi_r(x)\big| &+ \big| I_T(x,s)\big|
  \ge
  \big| S_n\vfi_r(x) + I_T(x,s)\big| \ge\epsilon T
\end{align*}
we have that for every $0<\xi<1$ the deviation set is
contained in
\begin{align}\label{eq:deviationset}
  \Big\{ (x,s)\in X_r : |S_n\vfi_r(x)|\ge\epsilon(1-\xi) T\Big\}
  \cup
  \Big\{ (x,s)\in X_r : |I_T(x,s)| \ge\epsilon\xi T\Big\}.
\end{align}
Observe first that by the exponential tail of $r$ the
following subset $R_L:=\{x\in X : r(x)>L\}$ for $L>0$
satisfies
\begin{align*}
  C_0:=\int e^{\epsilon_0 r}\,d\mu
  \ge
  \int_{R_L}e^{\epsilon_0 r}\,d\mu
  \ge
  e^{\epsilon_0 L}\mu(R_L)
  \quad\text{thus}\quad
  \mu(R_L)\le C_0 e^{-\epsilon_0 L}.
\end{align*}
Now taking $L>0$ big enough so that $(n+1) e^{-\epsilon_0
  n/2}<1$ for all $n>L$
\begin{align}
  \int_{R_L}\hspace{-0.2cm}r\,d\mu
  &
  \le
  \sum_{i\ge L}\int_{i}^{i+1} \hspace{-0.5cm}r\,d\mu
  \le
  C_0\sum_{i\ge L} (i+1) e^{-\epsilon_0 i}
  \le
  C_0\sum_{i\ge L} e^{-\epsilon_0 i /2}
    \le
  C_0\frac{e^{-\epsilon_0 L/2}}{1-e^{-\epsilon_0/2}}.\label{eq:int-r-geL}
\end{align}
Now we deduce an upper bound for the measure of each set
in~\eqref{eq:deviationset}.  On the one hand, writing
$\|\vfi\|$ for $\sup|\vfi|$, since
\begin{align*}
  |I_T(x,s)|
  \le
  \big(s+S_{n+1}r(x)-S_n r(x)\big)\cdot\|\vfi\|
  =
  \big(s+ (r\circ\sigma^n)(x)\big) \cdot\|\vfi\|
\end{align*}
we obtain, using that $\mu$ is $\sigma$-invariant and \eqref{eq:int-r-geL}
\begin{align}
  \mu_r\{(x,s)&\in X_r: |I_T(x,s)| \ge \epsilon\xi T\}\nonumber
  \\
  &\le
  \mu_r\Big\{ (x,s)\in X_r: s\ge\frac{\epsilon\xi T}{2\|\vfi\|} \Big\}
  +
  \mu_r\Big\{ (x,s)\in X_r: (r\circ\sigma^n)(x)\ge
  \frac{\epsilon\xi T} {2\|\vfi\|}\Big\}\nonumber
  \\
  &=\frac1{\ov{r}}\left(
  \int_{\{x\in X: r(x)\ge\epsilon\xi T/(2\|\vfi\|)\}} r \, d\mu
  +
  \int_{\{x\in X: (r\circ\sigma^n)(x)\ge\epsilon\xi T/(2\|\vfi\|)\}} r\circ
  \sigma^n \, d\mu\right)\nonumber
  \\
  &=\frac1{\ov{r}}\left(
  \int_{R_{\epsilon\xi T/(2\|\vfi\|)}} r\,d\mu +
  \int_{\sigma^{-n}R_{\epsilon\xi T/(2\|\vfi\|)}} r\circ
  \sigma^n\,d\mu \right)
  =\frac2{\ov{r}}\int_{R_{\epsilon\xi T/(2\|\vfi\|)}} r\,d\mu\nonumber
  \\
  &\le
  2\frac{C_0}{\ov{r}}\cdot\frac{e^{-\epsilon_0\epsilon\xi T
      /(2\|\vfi\|)}}{1-e^{-\epsilon_0/2}}.
    \label{eq:LT2}
\end{align}
On the other hand, there is a relation between $n(x,s,T)$
and $T$ for $\mu_r$ almost all points, where we write $\bar
r$ for $\mu(r)=\int r\,d\mu$
\begin{align}\label{eq:LD3}
  \frac{S_nr(x)}n
  \le
  \frac{T+s}{n}
  <
  \frac{S_{n+1}r(x)}n
  \quad\text{so}\quad
  \frac{n}T=\frac{n(x,s,T)}T\xrightarrow[T\to\infty]{}
  \frac1{\overline{r}}.
\end{align}
Note that the left hand side subset
in~\eqref{eq:deviationset} is contained in the following
union for all sufficiently small $a>0$
\begin{align}\label{eq:LD5}
 \Big\{(x,s)\in X_r: \frac{T}{n} \le(1-a)\overline{r}\Big\}
 \cup
 \Big\{(x,s)\in X_r:
 \big|S_n\vfi_r(x)\big|\ge
 n\epsilon(1-\xi)(1-a)\overline{r}\Big\},
\end{align}
where we are omitting the dependence of $n$ on $(x,s,T)$ for
simplicity.  Again given $\omega>0$ the right hand subset
in~\eqref{eq:LD5} is contained in
\begin{align}
  \label{eq:LDT1}
(X\setminus R_{\omega T})\cap\Big\{(x,s)\in X_r &:
\big|S_n\vfi_r(x)\big|\ge
n\epsilon(1-\xi)(1-a)\overline{r} \,
\&\, \frac{T}n\le (1+a)\overline{r}
\Big\}\nonumber
\\
&\cup
R_{\omega T}
\cup
\Big\{(x,s)\in X_r: \frac{T}n> (1+a)\overline{r}\Big\}.
\end{align}
For the first subset in~\eqref{eq:LDT1} we can use
Theorem~\ref{mthm:deviation-count-shift} (since we have a
$\sigma$-periodic point $z$ such that $S_p\vfi_r(z)\neq0$
from condition~\eqref{eq:periodicond} and from
Lemma~\ref{le:alaLivsic} we know that $\vfi_r$ is not
cohomologous to the zero function) to obtain an
exponentially small upper bound depending on $T$. We write
$R_L^c$ for $X\setminus R_L$ for any $L>0$ in what follows
\begin{align}
  \mu_r\big(R_{\omega T}^c\cap\Big\{
&\big|S_n\vfi_r\big|
\ge
n\epsilon(1-\xi)(1-a)\overline{r} \,\&\,
\frac{T}n\le (1+a)\overline{r}
\Big\}\big)\nonumber
\\
&\le
\frac{\omega T}{\ov{r}}\mu\Big\{x\in X:
\big|S_n\vfi_r(x)\big|
\ge
n\epsilon(1-\xi)(1-a)\overline{r}
\,\&\,
n \ge \frac{T}{(1+a)\overline{r}}
\Big\}  \nonumber
\\
&\le
\frac{\omega T}{\ov{r}} \cdot e^{(\beta+\delta) T /
  ((1+a)\overline{r})},
\label{eq:LDT2}
\end{align}
for some small $\delta>0$,
where 
$\beta=\beta(a,\xi)<0$ is given by
Theorem~\ref{mthm:deviation-count-shift}
\begin{align*}
  \beta= \sup_{\nu\in\M_\sigma}
  \Big\{h_\nu(\sigma)-\int \psi\,d\nu:
  |\nu(\vfi_r)|\ge\epsilon(1-\xi)(1-a)\overline{r} ,
   \psi\in
  L^1(\nu)\Big\}.
\end{align*}
For the middle subset in \eqref{eq:LDT1} we can use the
bound \eqref{eq:int-r-geL} to get
\begin{align}
  \label{eq:LDT21}
  \mu_r(R_{\omega T}) \le \frac{C_0}{\ov{r}}\frac{e^{-\epsilon_0\omega
      T/2}}{1-e^{-\epsilon_0/2}}.
\end{align}
Now we only need an upper large deviation estimate on
$n(x,s,T)/T$ to finish.

\subsubsection{The lap number versus flow time}
\label{sec:lap-number-versus}

From~\eqref{eq:LD3} we consider the measure of the following
subsets of $X_r$ for any given $0<\zeta<1/\ov{r}$
\begin{align}
  \mu_r\Big\{\ \big| \frac{n(x,s,T)}T-\frac1{\ov{r}}  \big| \ge \zeta
  \Big\}
  &=
  \mu_r\Big\{\ \frac{n}T-\frac1{\ov{r}} \ge\zeta
  \Big\}
  +
  \mu_r\Big\{\ \frac{n}T-\frac1{\ov{r}} \le -\zeta
  \Big\}\nonumber
  \\
  \text{(by inequality \eqref{eq:LD3})}
  &=
  \mu_r\Big\{\ T\le\frac{n\ov{r}}{1+\zeta\ov{r}} \,\,\&\,\,
  \frac1n S_nr \cdot
  \big(1-\frac{s}{S_n r}\big) \le \frac{\ov{r}}{1+\zeta\ov{r}}
  \Big\}\nonumber
  \\
  &\quad
  +
  \mu_r\Big\{\ \frac{T}{n}\ge \frac{\ov{r}}{1-\zeta\ov{r}}
  \Big\}.\label{eq:LD22}
\end{align}
Since $r$ itself can be taken as an observable in
Theorem~\ref{mthm:deviation-count-shift}, for $n$ so big
that
\begin{align*}
  1-\frac{s}{S_nr(x)}\ge 1-\frac{s}{nr_0} \ge 1-\xi>0
  \quad\text{with}\quad
  \frac{\ov{r}}{(1-\xi)(1+\zeta\ov{r})}<\ov{r}
\end{align*}
we can bound the first summand in~\eqref{eq:LD22} by
\begin{align*}
  \mu_r\Big\{\ T\le\frac{n\ov{r}}{1+\zeta\ov{r}} &\,\,\&\,\,
  \frac1n S_nr\cdot\big(1-\frac{s}{S_n r}\big)
  \le \frac{\ov{r}}{1+\zeta\ov{r}}
  \Big\}
  \\
  &\le
  \mu_r\underbrace{\Big\{
   n\ge T\frac{1+\zeta\ov{r}}{\ov{r}}
  \,\,\&\,\,
  \frac1n S_nr\le\frac{\ov{r}}{(1-\xi)(1+\zeta\ov{r})}\Big\}}_{A_n}.
\end{align*}
Now we split into pieces that are easier to estimate, for
$\omega>0$ small and $T$ big we have,
from~\eqref{eq:int-r-geL} and
Theorem~\ref{mthm:deviation-count-shift}
\begin{align}
  \mu_r(A_n)&=\mu_r(A_n\cap R_{\omega T})+\mu_r(A_n\setminus
  R_{\omega T})\nonumber
  \\
  &\le\nonumber
  \mu_r(R_{\omega T}) + \frac{\omega T}{\ov{r}} \mu\big\{x\in X:
   n\ge T\frac{1+\zeta\ov{r}}{\ov{r}}
  \,\,\&\,\,
  \frac1n S_nr\le \frac{\ov{r}}{(1-\xi)(1+\zeta\ov{r})}\big\}
  \\
  &\le
  \frac{C_0}{\ov{r}}\frac{e^{-\epsilon_0\omega T/2}}{1-e^{-\epsilon_0/2}}
  +
   \frac{\omega T}{\ov{r}} e^{(\gamma+\delta) (1+\zeta\ov{r}) T /\ov{r} },\label{eq:LD23}
\end{align}
because $(\gamma+\delta) n < (\gamma+\delta) (1+\zeta\ov{r}) T /\ov{r}$,
where $\delta>0$ is small and 
$\gamma=\gamma(\xi,\zeta)<0$ is given by
\begin{align*}
  \sup_{\nu\in\M_\sigma}\left\{h_\nu(\sigma)-\!\!\int\! \psi\,d\nu:
  |\nu(r)-\bar r|\ge\bar r\Big(1-\frac1{(1-\xi)(1+\zeta\ov{r})}\Big) ,
   \psi\in L^1(\nu)\right\}.
\end{align*}
For the second summand in~\eqref{eq:LD22} observe that,
using the relation~\eqref{eq:LD3} and considering
the position of $n\ov{r}/(1-\zeta\ov{r})$ on the real line
with respect to $S_nr(x)$ (see
Figure~\ref{fig:relative-positi-real}), we have either
\begin{align*}
  r(\sigma^n(x))&=S_{n+1}r(x)-S_nr(x)\ge D/2, \quad\text{or}
  \\r(\sigma^{n-1}(x))&=S_{n}r(x)-S_{n-1}r(x) \ge D/2,
\end{align*}
where
$D=T+s-\ov{r}n/(1-\zeta\ov{r}))$.
\begin{figure}[htpb]
  \centering
  \psfrag{S0}{$S_{n-1}r$}\psfrag{S1}{$S_{n}r$}\psfrag{S2}{$S_{n+1}r$}
  \psfrag{T}{$T+s$}\psfrag{D}{$D$}\psfrag{r1}{$n\ov{r}/(1-\zeta\ov{r})$}
  \includegraphics[width=8cm]{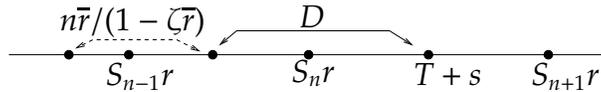}
  \caption{\label{fig:relative-positi-real}Relative
    positions on the real line of $T+s$ and
    $n\ov{r}/(1-\zeta\ov{r})$.}
\end{figure}

Then setting $\tau:=T/n>\ov{r}/(1-\zeta\ov{r})>\ov{r}>r_0$
we can write, by the $\sigma$-invariance of $\mu$ together
with the tail estimate~\eqref{eq:int-r-geL} and the bound
$T\ge r_0 n$ (recall that $n=n(x,s,T)$)
\begin{align*}
  \mu_r&\Big\{(x,s)\in X_r: \frac{T}{n(x,s,T)}\ge
  \frac{\ov{r}}{1-\zeta\ov{r}}
  \Big\}
  \\
  &\le
  \mu_r
  \Big\{ r\circ \sigma^{n-i}\ge
  T\left(1+\frac{s}T-\frac{\ov{r}/\tau}{1-\zeta\ov{r}}\right)\,\,\&\,\,
  \frac{T}n=\tau, \,\text{for}\, i=0,1\Big\}
  \\
  &\le
  \mu_r
  \Big\{ r\circ \sigma^{n-i}\ge
  \underbrace{T\left(1-\frac{\ov{r}/r_0}{1-\zeta\ov{r}}\right)}_{T(\zeta)}
  \quad\text{for }\, i=0 \,\text{ or }\, i=1\Big\}
  \\
  &=
  \mu_r\{(x,s)\in X_r:
  x\in\sigma^{-n} R_{T(\zeta)}\}
  +
  \mu_r\{(x,s)\in X_r:
  x\in\sigma^{-n+1} R_{T(\zeta)}\}
  \\
  &\le
  \sum_{k=0}^{[T/r_0]+1} \mu_r
  \underbrace{ \{(x,s)\in X_r: x\in \sigma^{-k}
  R_{T(\zeta)} \}}_{B_k}.
\end{align*}
Now we split the set in two parts as in \eqref{eq:LD23} and
use the $\sigma$-invariance of $\mu$
\begin{align}
  &\sum_{k=0}^{[T/r_0]+1} \big(\mu_r(B_k\cap R_{\omega T}) +
  \mu_r(B_k\setminus R_{\omega T})\big)\nonumber
  \\
  &\le
  \frac{C_0}{\overline{r}}\Big(\big[\frac{T}{r_0}\big]+2\Big)
  \frac{e^{-\epsilon_0\omega  T/2}}{1-e^{-\epsilon_0/2}}
  +
  \frac{\omega T}{\overline{r}}
  \sum_{k=0}^{[T/r_0]+1} \mu(R_{T(\zeta)})\nonumber
  \\
  &\le
    \frac{C_0+\omega T}{\overline{r}}\Big(\big[\frac{T}{r_0}\big]+2\Big)
    \left(
      \frac{e^{-\epsilon_0\omega T/2}+e^{-T\frac{\epsilon_0}{2}(1-\ov{r}/(r_0-r_0\zeta\ov{r}))}}
      {1-e^{-\epsilon_0/2}}
    \right).
  \label{eq:LD21}
\end{align}
Putting~\eqref{eq:LD23} and~\eqref{eq:LD21} together and
letting $\omega,\delta>0$ be arbitrarily small we get
\begin{align}
  \label{eq:limsup-lap}
  \limsup_{T\to+\infty}\frac1T\log\mu_r\Big\{\ \big|
  \frac{n}T-\frac1{\ov{r}} \big| \ge \zeta \Big\}
  \le
  \max\Big\{\gamma\frac{1+\zeta\ov{r}}{\ov{r}},
  -\frac{\epsilon_0}{2}
  \big(1-\frac{\ov{r}}{r_0(1-\zeta\ov{r})}\big)\Big\}.
\end{align}

\subsubsection{Exponentially small tail}
\label{sec:exponent-small-tail}

Finally, comparing the right hand subset in \eqref{eq:LDT1}
with the usage of $\zeta$ in \eqref{eq:LD22} of
Subsection~\ref{sec:lap-number-versus}, we see that
$a+1=(1-\zeta\ov{r})^{-1}$ thus
$\zeta=\frac{a}{1+a}\cdot\frac1{\ov{r}}$; so that putting
\eqref{eq:LT2},\eqref{eq:LDT1}, ~\eqref{eq:LDT2},
~\eqref{eq:LDT21} and \eqref{eq:limsup-lap}
together 
we arrive at (letting again $\omega,\delta>0$ be arbitrarily
small)
\begin{align*}
  \limsup_{T\to+\infty}&\frac1T\log\mu_r
  \Big\{z\in X_r:\big|\int_0^T
  \hspace{-0.2cm}\vfi\big(f_t(z)\big)\,dt\big| \ge \epsilon
  T\Big\}
  \\
  &\le \max\Big\{ \frac\beta{(1+a)\ov{r}},
  \frac{\gamma}{\ov{r}}\cdot\frac{2+a}{1+a},
  -\frac{\epsilon_0}2\big(1-\frac{\ov{r}}{r_0(1-\zeta\ov{r})}\big)
  , -\frac{\epsilon_0 \epsilon \xi}{2\|\vfi\|},
  -\frac{\epsilon_0 \omega}{2\|\vfi\|}
  \Big\}.
\end{align*}
for all small enough $a,\zeta>0$ and also
$\xi,\epsilon,\omega>0$.  Observe that $\epsilon_0$ does not
depend on $\epsilon$ and by the assumptions on $\mu$
(i.e. $\mu$ is the unique equilibrium state for the
potential $\psi$) we have
$\gamma(\xi,\zeta) \xrightarrow[\xi,\zeta\to0]{}0$.  Thus we
can take $\epsilon,\xi,\zeta>0$ so small that
$\beta/((1+a)\ov{r})$ is the maximum value above. Then
letting $a$ be very small we obtain the statement of
Theorem~\ref{mthm:devsemiflow}.

\subsection{The limit inferior}
\label{sec:limit-inferi}

For the limit inferior we need to restrict the class of
observables to consider. We assume that $\vfi:X_r\to\RR$ is
continuous and bounded,  with $\mu_r(\vfi)=0$ and
$\vfi_r:X\to\RR$ an H\"older function and, in addition, that $\vfi$ has
\emph{compact support}:
there exists a compact subset $K\subset X_r$ such that
$\vfi\equiv0$ on $X_r\setminus K$. Let $r_1=\max_K r\ge r_0$
in what follows.  We now show that any deviation set for
$\vfi$ under the flow $f_t$ can be related to a specific
deviation set for $\vfi_r$ under the shift map, in such a
way that we can apply the lower bound for the rate of large
deviations provided by
Theorem~\ref{mthm:deviation-count-shift}.

We start by noting that the function
\begin{align*}
  \varrho(x,s):= \vfi(x,s)- \vfi_r(x)
\end{align*}
is bounded and satisfies
\begin{align*}
  \varrho_r(x)=\int_0^{r(x)}
  \big(\vfi(x,t)-\vfi_r(x)\big)\,dt
  =\int_0^{r_1}\big(\vfi(x,t)-\vfi_r(x)\big)\,dt=0
\end{align*}
and
\begin{align*}
  \|\varrho\|
  &:=
  \sup_{(x,s)\in X_r}
  \left|\vfi(x,s)-\int_0^{r(x)}\vfi(x,t)\,dt\right|
  \le
  \|\vfi\|+r_1\|\vfi\|=(1+r_1)\|\vfi\|.
\end{align*}
Now from relation~\eqref{eq:timeTvsn} applied with $\varrho$
in the place of $\vfi$, for all $(x,s)\in X_r$ and $T>0$,
with $n=n(x,s,T)$
\begin{align}
  \int_0^T \varrho\big(f_t(x,s)\big) \, dt
  &=
  S_n\varrho_r(x)+I_T(x,s) = I_T(x,s)\quad\text{and} \nonumber
  \\
  |I_T(x,s)|
  &\le
  \left|\int_0^{T+s-S_n r(x)}
  \hspace{-0.7cm}\varrho\big(f_t(\sigma^n(x),0)\big)\,dt
  \right|
  +
  \hspace{-0.1cm}\nonumber
  \left|\int_0^s\hspace{-0.2cm}\varrho\big(f_t(x,0)\big)\,dt\right|
  \\
  &\le \label{eq:unifbound}
  C_1:=2r_1\|\varrho\|\le 2r_1(1+r_1)\|\vfi\|.
\end{align}
Therefore, by the definition of $\varrho$, for each
$(x,s)\in X_r$ and all $T>0$
\begin{align}\label{eq:liminf1}
  \Big|\int_0^T \vfi\big(f_t(x,s)\big) \, dt\Big|
  \ge
  \Big| \int_0^T \vfi_r\big(\pi\circ f_t(x,s)\big) \, dt \Big|
  - C_1,
\end{align}
where $\pi:X_r\to X$ is the projection on the first
coordinate.

We observe that, because $\vfi$ has compact support, using
the relation~\eqref{eq:timeTvsn}, \emph{for the purpose of calculating
$\int_0^T\vfi(f_t(x,s))\,dt$ with given $(x,s)\in X_r$ and
$T>0$, we may assume without loss
of generality that both $s<r_1$ and $T+s-S_nr(x)<r_1$, since
$\vfi(y,t)=0$ for all $y\in X$ and $t\ge r_1$.} In other
words, any value of the Birkhoff integral of $\vfi$ for the
flow $f_t$ always coincides with the value of the Birkhoff
integral for some $(x,s)\in X_r$ and $T>0$ satisfying the
conditions stated above.

Now we use again the relation~\eqref{eq:timeTvsn} with
$\vfi_r$ in the place of $\vfi$ to get
\begin{align}
  \Big|\int_0^T \hspace{-0.3cm} \vfi_r\big(\pi\circ f_t(x,s)\big) \, dt \Big|
  &=
  \Big|  \sum_{i=0}^{n-1} \hspace{-0.1cm}\int_0^{r(\sigma^i(x))}\hspace{-0.8cm}
  \vfi_r(\sigma^i(x)) \, dt  -s\vfi_r(x) \nonumber
  \\
  &\quad + \big(
  T-r(\sigma^{n-1}(x))\big)
  \vfi_r\big(\sigma^{n-1}(x)\big)\Big|\nonumber
  \\
  &\ge
  \big|S_n(r\cdot\vfi_r)(x)\big| -
  2r_1\|\vfi_r\|\nonumber
  \\
  &\ge
  r_0\cdot \big|S_n\vfi_r(x)\big| -
  2r_1^2\|\vfi\|. \label{eq:liminf2}
\end{align}
This implies that if
$\big|S_n\vfi_r(x)\big|>\epsilon(1+\xi)T/r_0$, then
\begin{align*}
  \Big|\int_0^T \vfi\big(f_t(x,s)\big) \, dt\Big|
  &\ge
  r_0 \cdot \frac{\epsilon(1+\xi)T}{r_0} - 2r_1^2\|\vfi\| - C_1
  \\
  &=
  \big( \epsilon(1+\xi) - \frac{2r_1^2+2r_1(1+r_1)}T \|\vfi\|\big)
  T > \epsilon T
\end{align*}
for all $\xi,\epsilon>0$ and
$T>(4r_1^2+2r_1)\|\vfi\|/(\xi\epsilon)$. Therefore for
$\epsilon,\xi,\zeta>0$ we can write
\begin{align*}
  \mu_r\Big\{
  (x,s)&: \Big|\int_0^T \hspace{-0.2cm}\vfi\big(f_t(x,s)\big) \,
  dt\Big|>\epsilon T
  \Big\}
  \ge
  \mu_r\big\{ (x,s):
  \big|S_n\vfi_r(x)\big|>\epsilon(1+\xi)T/r_0\big\}
  \\
  &\ge
  \mu_r\big\{ (x,s)\in X_r:
  \big|S_n\vfi_r(x)\big|>\epsilon(1+\xi)\frac{T}{r_0}
  \,\,\&\,\,
  \frac{T}n\le \frac{\overline{r}}{1-\zeta\overline{r}} \big\}.
\end{align*}
Finally, since $r\ge r_0$, we have the following (crude)
lower bound for the last expression
\begin{align*}
  r_0\cdot \mu\big\{ x\in X: \big|S_n\vfi_r(x)\big|>
  \frac{\epsilon(1+\xi)\overline{r}}{r_0(1-\zeta\overline{r})} n\big\}.
\end{align*}
From Theorem~\ref{mthm:deviation-count-shift} 
we obtain for all big enough $n$ and $T>0$ (recall that $T\ge n r_0$)
\begin{align*}
   \mu_r\big\{
  (x,s)\in X_r: \Big|\int_0^T \vfi\big(f_t(x,s)\big) \,
  dt\Big|>\epsilon T
  \big\}
  \ge
  r_0\cdot e^{(\omega+\delta) n}
  \ge
  r_0\cdot e^{(\omega+\delta) T/r_0}
\end{align*}
where $\delta>0$ can be taken arbitrarily small and
$\omega=\omega(\epsilon,\xi,\zeta)<0$ is given by
Theorem~\ref{mthm:deviation-count-shift}
\begin{align*}
  \omega=\sup\big\{ h_\nu(\sigma)-\int\psi \,d\nu:
  |\nu(\vfi_r)|>
  \frac{\epsilon(1+\xi)\overline{r}}{r_0(1-\zeta\overline{r})},
  \nu\in \M_\sigma,
  \psi\in L^1(\nu)\big\}.
\end{align*}
Since $\xi,\zeta,\delta>0$ are arbitrary, we see that the
exponential decay rate of the measure is bounded below by
\begin{align*}
  \liminf_{T\to+\infty}\frac1T\log
  \mu_r\big\{
  \Big|\int_0^T \vfi\big(f_t(x,s)\big) \,
  dt\Big|>\epsilon T
  \big\}
  \ge
  \frac{\omega(\epsilon,0,0)}{r_0}.
\end{align*}
The proof of Theorem~\ref{mthm:devsemiflow} is complete.


\section{Application to the Teichm\"uller flow}
\label{sec:exampl-applic}

In this short section we apply Theorem~\ref{mthm:devsemiflow} to the
coding of the Teichm{\"u}ller flow on the moduli space of abelian
differentials.

The applications of these results to systems admitting a
coding through flows over countable full shifts are
consequences of the following simple observation.

We recall that a measure preserving dynamical system
$(Y,g_t,\B,\nu)$ (where $g_t$ is a $\B$-measurable flow) is
a \emph{factor} of the system $(X,f_t,\A,\mu)$ (where $f_t$
is a $\A$-measurable flow) if:
\begin{itemize}\item
  there exists a measurable map $\mathbf{i}:Y\to X$ which commutes
  with the actions of the dynamical systems:
  $\mathbf{i}(g_ty)=f_t(\mathbf{i}y)$ for all $y\in Y$ and all $t$;
\item $\mathbf{i}(Y)=X$ and the induced measure $\nu(\mathbf{i}^{-1}A),
  A\in\A$ equals $\mu$.
\end{itemize}

\begin{lemma}
  \label{le:isomorph}
  Let us assume that $(Y,g_t,\A,\nu)$ is a factor of $(X,f_t,\B,\mu)$
  with a factor map $\mathbf{i}:Y\to X$.

  Then for any observable $\vfi:X\to\RR$ with $\mu(\vfi)=0$ we have
  that the deviation sets
  \begin{align*}
    D_X(\vfi,\epsilon)&=\Big\{z\in X:\big| \int_0^T
    \vfi\big(f_t(z)\big) \,dt \big|>\epsilon T\Big\}\text{  and  }
    \\
    D_Y(\vfi\circ \mathbf{i},\epsilon)&=\Big\{z\in Y:\big| \int_0^T
    \big(\vfi\circ \mathbf{i}\big)\big(f_t(z)\big) \,dt \big|>\epsilon T\Big\}
  \end{align*}
  are related as follows:
  \begin{align*}
        D_X(\vfi,\epsilon)&=h\big(
    D_Y(\vfi\circ \mathbf{i},\epsilon)\big)
  \end{align*}
\end{lemma}

So if we can relate two flows as above and identify the
class of functions $\psi$ such that there exists
$\vfi:X\to\RR$ satisfying $\psi=\vfi\circ \mathbf{i}$ and a large
deviation estimate for the system $(X,f_t,\B,\mu)$, then
we can pass the same estimates for that class of functions
on the system $(Y,g_t,\A,\nu)$.

\begin{remark}\label{rmk:caveat}
  However if the given isomorphism does not respect other
  measures, then we may not be able to interpret the
  deviation rates for the system $(Z,Y^t,m)$ as the
  variational bounds in
  Theorems~\ref{mthm:deviation-count-shift}
  and~\ref{mthm:devsemiflow}. See item (3) of
  Proposition~\ref{pr:codingmaps}.
\end{remark}

We note that the roof functions $r_n:X\to\RR_+$ in
Proposition~\ref{pr:codingmaps} are H\~older and bounded
away from zero, so they are
automatically $\log$-H\"older as well: if $r_n(\omega)\ge
b_n >0$ for all $\omega\in X$, then for $N\in\NN$,
$\omega,\omega^\prime\in X$ with $\omega^\prime\in [\omega]_N$
\begin{align*}
  \left| 1 - \frac{r_n(\omega)}{r_n(\omega^\prime)}\right|
  = \frac{|r_n(\omega)-
    r_n(\omega^\prime)|}{|r_n(\omega^\prime)|}
  \le
  \frac{C\alpha^N}{b_n}.
\end{align*}
Moreover, fixing $n\in\NN$ and connected component $\HH$ of
$\modk$, a function $\vfi:\HH\to\RR$ which is bounded and
H\"older in the sense of Veech induces a function
$\theta:\VR\to\RR$ so that $\vfi\circ \pi_{\cal R} =
\theta$, and then the function $\psi=\theta\circ
\mathbf{i_n}:X_{r_n}\to\RR$ is such that
$\psi_{r_n}:X\to\RR$ is H\"older.

Finally, each roof function $r_n$ has exponential tail (with
respect to $\mu_\kappa$) since $\tau_K\ge c$ for some
positive constant $c$ for the compact $K\subset\HH$ in
\eqref{expretest}.

Hence, we can use Theorem~\ref{mthm:devsemiflow} with
$\psi$ as the observable to estimate the rate of decay of
the deviation sets for $\vfi$.

\def\cprime{$'$}


\end{document}